\documentclass[a4paper]{amsart}
\numberwithin{equation}{section}

\usepackage{amssymb}
\usepackage{graphicx}
\usepackage{epsf}
\usepackage{color}

\newcommand{\RR}{{\rm\bf R}}
\newcommand{\ZZ}{{\rm\bf Z}}
\newcommand{\NN}{{\rm\bf N}}
\newcommand{\ee}{{\rm e}}
\newcommand{\pp}{{\rm \bf p}}
\newcommand{\Euc}{{\rm\bf E}}
\newcommand{\Ort}{{\rm\bf O}}
\newcommand{\Gtil}{\widetilde \Gamma}
\newcommand{\ptg}{{\rm\bf J}}
\newcommand{\ptgtil}{{\rm\bf \widetilde J}}
\newcommand{\cal}{\mathcal}
\newcommand{\LL}{{\cal L}}
\newcommand{\Ltil}{{\widetilde{\cal L}}}
\newcommand{\Ld}{{\cal L}^*}
\newcommand{\Ltily}{{\widetilde{\cal L}}_{y_0}}
\newcommand{\MM}{{\cal M}}
\newcommand{\proj}{{\cal P}}
\newcommand{\restr}{{\cal R}}
\newcommand{\hex}{{\cal H}}
\newcommand{\OO}{{\cal O}}
\newcommand{\faixa}{{\rm\bf B}_{y_0}}
\newcommand{\oa}{{\rm\bf a}}
\newcommand{\ob}{{\rm\bf b}}
\newcommand{\oc}{{\rm\bf  c}}
\newcommand{\na}{{\rm\bf \hat a}}
\newcommand{\nb}{{\rm\bf \hat b}}
\newcommand{\plano}{{\rm \bf R}^2}
\newcommand{\seta}{\longrightarrow}

\newtheorem{lema}{Lemma}[section]

\newtheorem{proposicao}{Proposition}[section]

\newtheorem{teorema}{Theorem}[section]

\begin{document}
\title[Periodic Functions, Lattices and Projections]{Periodic Functions, Lattices and Their Projections}
\author[I. S. Labouriau,  E. M. Pinho]{
Isabel S. Labouriau, Eliana M. Pinho}
\address{Isabel S. Labouriau --- Centro de Matem\'{a}tica da
   Universidade do Porto\\
   Rua do Campo Alegre, 687\\
   4169-007 Porto\\
Portugal}
\email{islabour@fc.up.pt}
\address{E. M. Pinho --- Centro de Matem\'{a}tica da
   Universidade do Porto\\
   Rua do Campo Alegre, 687\\
   4169-007 Porto\\
Portugal}
\email{empinho@gmail.com}


\begin{abstract}
Functions whose symmetries form a crystallographic group in particular have a lattice of periods, and the set of their level curves forms a periodic pattern.
We show how after projecting these functions, one obtains new functions with a lattice of periods that is not the projection of the initial lattice.
We also characterise all the crystallographic groups in three dimensions that are symmetry groups of patterns whose projections have periods in a given  two-dimensional lattice. 
The particular example of patterns that after projection have a hexagonal lattice of periods is discussed in  detail.
\end{abstract}

\maketitle

 \section{Introduction}
 Patterns in biology or chemistry 
 may arise as a consequence of spatial variations in concentration of one or more substances.
 In many cases the consequences of the variation may  only be observed in projections.
 Examples are shapes in animal coats and  reactions in thin layers of gel.  
 They are modelled by 
 partial differential equations, as for instance in \cite{DionneGolubitsky, GolubitskyStewart}, with a prescribed symmetry group.
  In this context,  we model a periodic pattern by interpreting it as the level sets of a periodic function. 
  The set of all periods is viewed as part of the group of symmetries of the function.
  
The projection of  periodic objects
into a subspace has many applications.
Both the object to be projected and the projection method depend on the context.
In classical crystallography, points in a unit cell of a lattice are projected along certain special directions.
The symmetries of these projections are described in the International Tables of Crystallography \cite[Section 2.2.14]{int-tab-a},
where it says
``Even though the projection of a finite object along any direction may be useful, the projection of a periodic object such as a crystal structure is only sensible along a rational lattice direction (lattice row). Projection along a nonrational direction results in a constant density in at least one direction.''

In contrast, there is no intrinsic choice of direction of projection when observing a three dimensional pattern.

Looking at periodic patterns is not the same as looking at the lattice of their periods, specially when a projection is involved. 
The main goal of this article is to establish this fact and quantify it.
By a lattice we mean a subset of $\RR^{n+1}$ that is a $\ZZ$-module with $n +1$ linearly independent  generators.
When we project a lattice we may obtain a set with too many generators.
When we project periodic functions we may obtain functions with too few periods.
For planar lattices and wallpaper patterns this follows from the results of \cite{LabouriauPinho,LabouriauPinho2}, here we treat the general case.

We start by looking at the projection into $\RR^n$ of a lattice in $\RR^{n+1}$.
After establishing some notation in Section~\ref{secLattices}, we 
discuss in Section~\ref{seccao-periodos-redes} some properties of the projection for general lattices, and we show how these properties are modified if we make the common assumption that the lattice is an integral lattice.

From Section~\ref{SecPatterns} onwards, we project  patterns  in $\RR^{n+1}$  into patterns in  $\RR^n$. 
General information on symmetric patterns is given in Section~\ref{SecPatterns}.
Given the  crystallographic  group $\Gamma$ of symmetries of the original pattern,
we describe in 
Section~\ref{seccao-periodos-funcoes} the symmetry group of the projected pattern and the way it depends on the width of the projection band and on the symmetries of the original pattern, using results of \cite{PinhoLabouriau}.
In particular we obtain all the periods of the projection and compare them to the periods of the restriction of the pattern and to the projection of the lattice.
This particular group of examples is interesting because it has been proposed as an explanation for certain  special patterns,  called black-eye patterns, as discussed in \cite{CLO,Gomes}.

 While Section~\ref{seccao-periodos-funcoes}  contains a discussion of how the periodicity properties of the projected pattern  varies with the projection width, for different types of symmetry groups,
Section~\ref{secProjectGeneral} addresses the converse question: which crystallographic groups in three dimensions are the symmetry groups of functions that after projection over a given band have a given lattice of periods?
We provide a complete classification, and use this to obtain, in Section~\ref{SecProjectionToHex}, information on the symmetries of functions that after projection have periods in a hexagonal lattice.

We determine in  Section~\ref{SecProjectionToHex} which are the spaces of symmetric functions whose projections have a hexagonal lattice of periods: they are spaces of functions with periods on a list of fourteen possible lattices.
Of these, only four types yield a hexagonally periodic projection for all widths of the projecting band.
Three of these have been reported in \cite{Oliveira}.
 Patterns for the fourth type are provided in Section~\ref{SecDiscussion}, along with a discussion of these examples.
The patterns illustrate well the fact that the symmetries of the lattice of periods are not necessarily symmetries of patterns with these periods.


 \section{Lattices and groups}\label{secLattices}
In this section we establish some of the terminology for the remainder of the article.
We use the notation
$(x,y) \in \RR^{n+1}, \mbox{ with } x \in \RR^n \mbox{ and }
y \in \RR.$
The reader is referred to Armstrong~\cite[chapters
24, 25 and 26]{Armstrong}, for results on Euclidean and
plane crystallographic groups, and to Senechal~\cite[chapter
2]{Senechal} and Miller~\cite[chapter 2]{Miller} for results
on lattices and crystallographic groups. A
detailed description can also be found in Pinho~\cite[chapter
2]{Pinho}. 

\subsection{the Euclidean group}

The $(n+1)$-dimensional {\em Euclidean group} is the
semi-direct product
$\Euc(n+1) \cong \RR^{n+1}
\ltimes \Ort(n+1)$ with elements
$\gamma=(v,\delta)$, where $v \in \RR^{n+1}$ and $\delta \in
\Ort(n+1)$.
The group operation is $(v_1,\delta_1) \cdot (v_2,\delta_2) =
(v_1 + \delta_1 v_2, \delta_1 \delta_2)$,
for $(v_1,\delta_1)$, $(v_2,\delta_2) \in \Euc(n+1)$ and
the {\em action} of $(v,\delta) \in
\Euc(n+1)$ on $(x,y) \in \RR^{n+1}$ is given by
$(v,\delta)\cdot (x,y)=v+\delta(x,y)$.

\subsection{lattices}

Using the definition of Senechal~\cite[page 37]{Senechal},
a subset $\LL \subset \RR^{n+1}$ is a {\em
lattice}\index{lattice} if it is generated over the integers
by $n+1$ linearly independent elements
$l_1, \ldots , l_{n+1} \in
\RR^{n+1}$, which we write:
$$\LL = \{ l_1, \ldots , l_{n+1} \}_\ZZ
= \left\{ \sum_{i=1}^{n+1}m_i l_i , m_i
\in \ZZ \right\}.$$ 

Any set of vectors $l_1, \ldots , l_{n+1}$ that generate
$\LL$ over the integers defines a $(n+1)$-dimensional parallelepiped
called the {\em fundamental cell}\index{fundamental cell} of
the lattice
and its volume $\rho$ is an invariant of
the lattice (see Senechal~\cite[page 38]{Senechal})

For any $l \in \LL$ there is some $m \in \ZZ$ such that $\frac{1}{m}l$
is the smallest element of $\LL$ colinear with $l$. Then
$\frac{1}{m}l$ is a generator of $\LL$, {\sl i.e.}, there are elements
$g_1, \ldots ,g_n \in \LL$ such that $\LL = \{\frac{1}{m}l , g_1,
\ldots ,g_n \}_\ZZ $.

The {\em
symmetry group} $\Gamma$ of a lattice
$\LL\subset\RR^{n+1}$ is the largest subgroup of $\Euc(n+1)$ that leaves
$\LL$ invariant, i.e. $\gamma\in\Gamma$ if and only if
$\gamma\cdot\LL=\LL$. 
If $(v,\alpha)\in\Gamma$ for some $\alpha\in \Ort(n+1)$ then
$(v,\alpha) \cdot (0,0)=v\in\LL$. 
Since $-v$ is also in $\LL$ then $(-v, I\!d_{n+1})\cdot(v,\alpha)=(0,\alpha)\in 
\Gamma$ and $\alpha \in {\rm\bf H}$, the largest subgroup of $\Ort(n+1)$ 
that leaves $\LL$ invariant, called the {\em holohedry} of $\LL$. For the symmetry group of a lattice $\LL$, we have $\Gamma=\LL \wedge {\rm\bf H}$.

The {\em dual lattice} $\Ld$ of $\LL$ is defined as:
$$\Ld = \{ k \in \RR^{n+1} : \langle k,l_i\rangle \in \ZZ, \;
i=1,\ldots,n+1 \}$$
where $\langle \cdot , \cdot\rangle$ is the usual inner product in
$\RR^{n+1}$.
It may be written as 
$\Ld = \{ l^*_1, \ldots ,l^*_{n+1} \}_\ZZ,$
where $l^*_i \in \RR^{n+1}$ and $\langle l^*_i,l_j\rangle=\delta_{ij}$ for
all $i,j \in \{1, \ldots , n+1 \}$.

\subsection{crystallographic groups}

A subgroup
$\Gamma \leq \Euc(n+1)$ is a {\em crystallographic group}
with {\em lattice} $\LL$ if the orbit, on
$\RR^{n+1}$, of the origin under the action of its {\em subgroup
of translations}
$\{v : (v, I\!d_{n+1}) \in \Gamma\}$, is a
lattice $\LL\in\RR^{n+1}$ (or a  $\ZZ$-module).
We also use the symbol $\LL$ for the
subgroup of translations of $\Gamma$, since it is isomorphic
to the group $(\LL, +)$. 

The projection $(v,\delta)  \longmapsto  \delta$, of $\Gamma$
into $\Ort(n+1)$, has kernel $\LL$ and image
$\ptg=\{ \delta : (v,\delta) \in \Gamma \mbox{ for some }
v\in \RR^{n+1} \}$,
 isomorphic to the quotient $\Gamma / \LL$.
The group $\ptg$ is called
the {\em point group} of $\LL$
and is a subgroup of the holohedry of
$\LL$. 
Thus, $\ptg \LL = \{\delta l : \delta \in \ptg , l \in \LL
\}=\LL$.
A caveat here: Armstrong and Senechal call this a point group. Miller and crystallographers use point group for $\Gamma\cap \Ort(n+1)$.

We will abuse terminology and also refer to 
the projection into $\Ort(n)$ of any subgroup of $\Euc(n)$ whose translations form a $\ZZ$-module with $m$ generators, $m\leq n$ as a  {\em point group}, even if the original group is not crystallographic.

The set of all the elements in $\Gamma$ with a given {\em orthogonal
component} $\delta \in \ptg$ is the coset $\LL
\cdot (v,\delta)=\{(l+v,\delta) : l\in \LL\}$ for any $v\in 
\RR^{n+1}$ such that $(v, \delta) \in \Gamma$.
We will denote $v+\LL$ the
{\em non-orthogonal component} of $(v,\delta)\in \Gamma$
defined up to elements of
$\LL$.

\subsection{special symmetries}
For $\alpha \in \Ort(n)$, we use the following notation for some  elements of $\Ort
(n+1)$:
$$\sigma=\left(
\begin{array}{cc}I\!d_n&0\\0&\
-1\end{array}\right),
\quad \alpha_+=\left(
\begin{array}{cc}\alpha&0\\0&\ 1\end{array}\right)
\quad \mbox{and} \quad \alpha_-=\sigma \alpha_+ =\left(
\begin{array}{cc}\alpha&0\\0&\ -1\end{array}\right)$$
where $I\!d_n$ is the $n\times n$ identity matrix.


\section{periods of projected
lattices\label{seccao-periodos-redes}}

We start by looking at the projection of the portion of a lattice lying  between two parallel
affine subspaces. 
 Let 
 $\proj: \RR^{n+1}\seta\RR^{n}$ be the projection $\proj(x,y)=x$ and let 
 $\faixa$ be the strip
  $$
\faixa=\left\{
(x,y)\in\RR^{n+1} :   0\le y \le y_0 \right\}.
$$
The projection of $\LL\cap\faixa$, given by
 $$
\proj(\LL\cap\faixa)=\left\{
x\in\RR^n : \exists y,  0\le y \le y_0, \mbox{ such that } 
(x,y)\in\LL
\right\}
$$
is not necessarily a lattice. 
We want to describe the set of its periods, given by
$$
\left\{ \pp\in\RR^n: \ x+\pp\in \proj(\LL\cap\faixa)\ \forall x\in \proj(\LL\cap\faixa)\right\}.
$$
Clearly, if $(\pp,0)\in\LL$, then $\pp$ is a period of $\proj(\LL\cap\faixa)$.

First we obtain conditions for periodicity involving special forms for the generators of $\LL$, 
that we call  {\em adapted generators}.

\begin{lema}\label{lemaGeradoresAdaptados}
Let $\LL$ be a lattice in $\RR^{n+1}$, with $(0,b)\in\LL$, $b\ne 0$, minimal in its direction.
Given $(\alpha,\beta)\in\LL$, $\alpha\ne 0$, there is a set of generators for $\LL$ of the form
$\{ (0,b),(a_i,b_i), \  i=1,\ldots,n\}$ with $ma_1=\alpha$ for some $m\in\ZZ$, and such that $0\le b_i<b$ for $i=1,\ldots,n$ and $\{a_i, i=1,\ldots,n\}$ are linearly independent vectors in $\RR^n$.
\end{lema}

\begin{proof}
We can write $\LL=\{ (0,b),(c_i,d_i), \  i=1,\ldots,n\}$ where  $\{c_i, i=1,\ldots,n\}$ are linearly independent vectors in $\RR^n$.
Let $W=\{ (c_i,d_i), \  i=1,\ldots,n\}_\RR\subset \RR^{n+1}$ and $\LL_2=\LL\cap W=\{ (c_i,d_i), \  i=1,\ldots,n\}_\ZZ\subset W$.
Then $(\alpha,\beta^\prime)\in \LL_2$ for some $\beta^\prime$.

Let $\LL_1$ be the lattice $\LL_1=\{ c_i, \  i=1,\ldots,n\}_\ZZ\subset \RR^{n}$,  with $\alpha\in\LL_1$.
Take $a_1$ to be the minimal elemento of $\LL_1$ in the direction of $\alpha$, i.e.
$a_1\in\LL_1$ and $ma_1=\alpha$ for some $m\in\ZZ$, $m_*>0$, and hence $a_1=\sum_{k=1}^{n} n_{1k}c_k$.
Then $(a_1,b_1)\in\LL_2$ for $b_1=\sum_{k=1}^{n} n_{1k}d_k$.

Taking generators $\{a_1,\ldots,a_n\}$ for the  lattice $\LL_1\subset \RR^n$, with $a_i=\sum_{k=1}^{n} n_{ik}c_k$, let $b_i\sum_{k=1}^{n} n_{ik}d_k$, then $\{(a_i,b_i), \  i=1,\ldots,n\}_\ZZ\subset \LL_2$.
The proof will be complete if we show that $\LL_2\subset\{(a_i,b_i), \  i=1,\ldots,n\}_\ZZ$.
For this it is sufficient to show that $(c_j,d_j)\in\{(a_i,b_i)\}_\ZZ$.

By construction, $c_j=\sum_{i=1}^n m_{ji}a_i=\sum_{i=1}^n \sum_{j=1}^n m_{ji}n_{ik}c_k$, hence
$\left(m_{ji}\right)\left(n_{ik}\right)=I\!d_{n\times n}$.
Then,
$$\sum_{i=1}^n m_{ji}b_i=\sum_{i=1}^n \sum_{j=1}^n m_{ji}n_{ik}d_k=d_j
\qquad\Rightarrow\qquad
(c_j,d_j)=\sum_{i=1}^n m_{ji}(a_i,b_i) .
$$
Finally, if not all $b_i<b$, we may change the set of generators to $\{(a_i,b_i-s_ib)\}$ where $s_i$ is the largest integer such that $s_ib<b_i$. 
This is still a set of generators, with the required properties.
\end{proof}

\begin{teorema}\label{ThDiophantine}
The projection $\proj(\LL\cap\faixa)$ of a lattice $\LL$  in
$\RR^{n+1}$  has period $\pp\in\RR^n-\{0\}$ 
if and only if either $(\pp,0)\in\LL$ or $(0,b)$ and 
 $(a_1,b_1) \in\LL$
with $m_*a_1=\pp$,  $m_*\in\ZZ$, $b,b_1\in\RR$,  $b>0$, and
in the second case,  for adapted generators $\{(0,b_0),(a_i,b_i),\ i=1,\ldots, n\}$
 as in Lemma~\ref{lemaGeradoresAdaptados},
the following diophantine condition holds:
for any $m_1,\ldots,m_n\in\ZZ$
 \begin{equation}\label{diophantineC}
 \sum_{i=1}^{n} m_ib_i\in[0,y_0]\pmod{b_0}
 \quad\Longrightarrow\quad
m_*b_1+  \sum_{i=1}^{n} m_ib_i\in[0,y_0]\pmod{b_0}
 \end{equation}
\end{teorema}

\begin{proof}
If $\proj(\LL\cap\faixa)$ is $\pp$-periodic,
since $0=\proj((0,0))\in\proj(\LL\cap\faixa)$, then   $\pp\in\proj(\LL\cap\faixa)$.
Therefore there exists $q\in[0,y_0]$ such that $(\pp,q)\in\LL$.
If $q=0$ we are in the first case.

If $q\ne 0$ and if
 $0<q\le y_0$, let $(\tilde{x},\tilde{y})=(n\pp,nq)\in\LL$ where
$n$ is the largest integer such that $nq\le y_0$.
Then both $\tilde{x}$ and $\tilde{x}+\pp\in\proj(\LL\cap\faixa)$ with
$(\tilde{x}+\pp,\tilde{y}+q)\in\LL$, $\tilde{y}+q\ge y_0$ and 
$(\tilde{x}+\pp,\tilde{z})\in\LL$ for some $\tilde{z}$ satisfying
$0\le\tilde{z}\le y_0$.
It follows that $(0,b)=(0,\tilde{y}+q-\tilde{z})\in\LL$.

Thus it remains to show that if $(\pp,0)\notin\LL$
and if $(0,b)$ and $(\pp,q) \in\LL$ then $\pp$ is a period of $\proj(\LL\cap\faixa)$
if and only if the diophantine condition (\ref{diophantineC})  holds for generators 
$(0,b_0)$, and $(a_i,b_i)$, $i=1,\ldots,n$,
adapted to
$(0,b)$ and $(\pp,q)$,
 i.e. with $n_* b_0=b$, and  $m_*(a_1,b_1)=(\pp,q)$, $n_*, m_* \in\ZZ$.

A point $x$ is in $\proj(\LL\cap\faixa)$ if and only if $(x,y)\in\LL$ with 
$x= \sum_{i=1}^{n} m_ia_i$ and 
$y= \sum_{i=1}^{n} m_ib_i\in[0,y_0]\pmod{b_0}$.
Since $\pp=m_*a_1$ then 
$x+\pp= \sum_{i=1}^{n} m_ia_i+m_*a_1$.
Thus $x+\pp\in\proj(\LL\cap\faixa)$ if and only if
$\tilde{y}=m_*b_1+ \sum_{i=1}^{n} m_ib_i\in[0,y_0]\pmod{b_0}$.
\end{proof}

The conditions of Theorem~\ref{ThDiophantine}
are easy to check in some cases, as in the next result.
The period may also be obtained from the symmetries of $\LL$:

\begin{proposicao}\label{propLatt}
Let $\LL$ be a lattice   in $\RR^{n+1}$ with symmetry group $\Gamma\subset \Euc(n+1)$.
The projection $\proj(\LL\cap\faixa)$  has period $\pp\in\RR^n-\{0\}$ if one of the following conditions holds:
\begin{enumerate}
\renewcommand{\theenumi}{\Roman{enumi}}
\item\label{Latt1} 
$(\pp,0)\in\LL$;
\item\label{Latt2} 
$(\pp,q)$ and $(0,b)\in\LL$
for some  $q,b\in\RR$, $0<b\le y_0$;
\item\label{Latt3} 
$\sigma \in \ptg$ 
and $(\pp,q) \in \LL$
for some $q\in\RR$, 
$0<q\le y_0$.
\end{enumerate}
The set $\proj(\LL\cap\faixa)$ is a lattice in Cases (\ref{Latt2}) and (\ref{Latt3}).
\end{proposicao}
The conditions of  Proposition~\ref{propLatt} are illustrated in Figures~\ref{figura-periodos-1}, \ref{figura-periodos-3} and
\ref{figura-periodos-2}.

\begin{proof}

Case (\ref{Latt1}) follows immediately from  Theorem~\ref{ThDiophantine}.
For Case (\ref{Latt2}), the diophantine condition \eqref{diophantineC} always holds because if $0<b_0\le b\le y_0$ then
for every $x\in\RR$ $\exists \tilde{x}\in [0,y_0]$ such that $x=\tilde{x}\pmod{b_0}$.

For Case (\ref{Latt3}), since $\sigma\in\ptg$ then $\sigma\LL\subset\LL$ and hence for any $(x,y)\in \LL$ we have
$(x,y)-\sigma(x,y)=(0,2y)\in \LL$.
In particular, $(0,2q)\in\LL$, let $(0,b)$, $0<b\le2q$ be the smallest  non-zero element in this direction.

If $b\le y_0$ we are in  Case (\ref{Latt2}).
Otherwise, $b=2q$ and $q\in(y_0/2, y_0]$. 
Let $(x,y)\in\LL\cap\faixa$. 
If $y=0$ then $(x,y)+(\pp,q)=(x+\pp ,q)\in\LL\cap\faixa$ and $x+\pp\in\proj(\LL\cap\faixa)$.
Otherwise, $0<y\le y_0$ and  for some $n\in\ZZ$ we have $2y=n_*b$ for some $n_*\in\ZZ$, hence $y=n_*q\le y_0$, hence $n_*=1$ and $y=q$.
It follows that $(x,y)+(\pp,q)-(0,2q)=(x+\pp,0)\in\LL$ hence $\pp$ is a period.
\end{proof}

Note that in  Case (\ref{Latt3}) of Proposition~\ref{propLatt}, from
$(\pp,q)\in\LL$  we  also get
$(2\pp,0)=(\pp,q)+\sigma(\pp,q)\in\LL$ yielding a larger period.

The conditions of Proposition~\ref{propLatt} are not necessary:
 if $\proj(\LL\cap\faixa)$  has period $\pp\in\RR^n-\{0\}$
and if  $(\pp,q)$ and $(0,b)\in\LL$ with $q>0$ and  $b>y_0$,
this does not entail that $\sigma$ is in the holohedry of
$\LL$, as the following example shows.

Let
$$
\LL=\left\{ (5,4),(0,7)\right\}_\ZZ\qquad y_0=6.$$
For each $n\in\ZZ$ there is $m\in\ZZ$ such that $4n+7m\in[0,6]$ because $4n\equiv 0,1,\ldots,6\pmod{7}$ always has a solution.
Fom this  it follows that 
for each $n\in\ZZ$ there is $m\in\ZZ$ such that
$n(5,4)+m(0,7)\in\faixa$ and therefore
 $\proj(\LL\cap\faixa)$ is 5-periodic.
 On the other hand $(5,-4)=\sigma(5,4)\not\in\LL$, 
 so  $\sigma$ is not in the holohedry of $\LL$.

 \subsection{Integral lattices}\label{subIntegralL}
Lattices where all the elements have integer
squared length are called  {\em integral lattices} and are important
 in the context of quasicrystals and quasiperiodic tilings --- see Senechal~\cite{Senechal} and Janssen {\sl et al}~\cite{JJ} for
 details.

 Given a subset $A\subset\RR^{n+1}$ and $y\in\RR$ let 
$$\restr_y(A)=\left\{ x:\ (x,y)\in A\right\}=
 \proj\left(A\cap\left(\RR^n\times \{y\}\right) \right).$$

Let $y_0=\max\{ y:\ (x,y)\in V(0)\}$,  the height of the {\em Vorono\"\i{} cell}
$$
V(0)=\left\{ (x,y)\in\RR^{n+1}:\ \left| (x,y)\right|\le \left| (x,y)-(p,q)\right| \ \forall (p,q)\in \LL\right\} .
$$
The restriction of the projection $\proj$ to $\LL\cap(B_{y_0}\cup \sigma B_{y_0})$ is called a {\em canonical projection}
when $\restr_0(\LL)=\{0\}$ .

The conditions in Proposition~\ref{propLatt} are also necessary in the more restrictive context of a canonical projection of integral lattices, see Senechal~\cite{Senechal}.
 
If $\LL$ is an integral lattice then the projection of $\LL\cap\faixa$ is non-periodic if 
$\restr_0(\LL)=\{0\}$, 
by Proposition~2.17 in Senechal~\cite{Senechal}.
It follows that 
in Case~\ref{Latt2} of Proposition~\ref{propLatt},  if $\restr_0(\LL)=\{0\}$ 
then $\LL$ is not an integral lattice.
To see this directly, let $(c,d)$ be any element of
$\LL$ and
$(0,a)\in
\LL$. If $\LL$ is an integral lattice then $a^2 \in \ZZ$ and
$\LL\subset \Ld$, see Senechal~\cite[section 2.2]{Senechal}.
Thus, $\langle (0,a),(c,d)\rangle \in \ZZ$ which implies $ad=n \in \ZZ$.
 Therefore, 
$a^2(c,d)-n(0,a)=(a^2c,0) \in \LL$ and it follows that the
restriction of $\LL$ to the subspace $y=0$ is not the origin
alone.
In Case~\ref{Latt3} of Proposition~\ref{propLatt},  we always have $\restr_0(\LL)\ne\{0\}$,
as we remarked before.

Figures~\ref{figura-periodos-1}, \ref{figura-periodos-3} and
\ref{figura-periodos-2} present some examples of projection of lattices in
$\RR^2$, illustrating the  cases in Proposition~\ref{propLatt}. 
The intersection of the lattice with the subspace
$y=0$ allows us to compare Proposition~\ref{propLatt} to results
on integral lattices. Figure~\ref{figura-periodos-4}
explains the periodicity of the projected lattice when
$(0,y_0) \in \LL$.

\begin{figure}[h]
\includegraphics[scale=0.3]{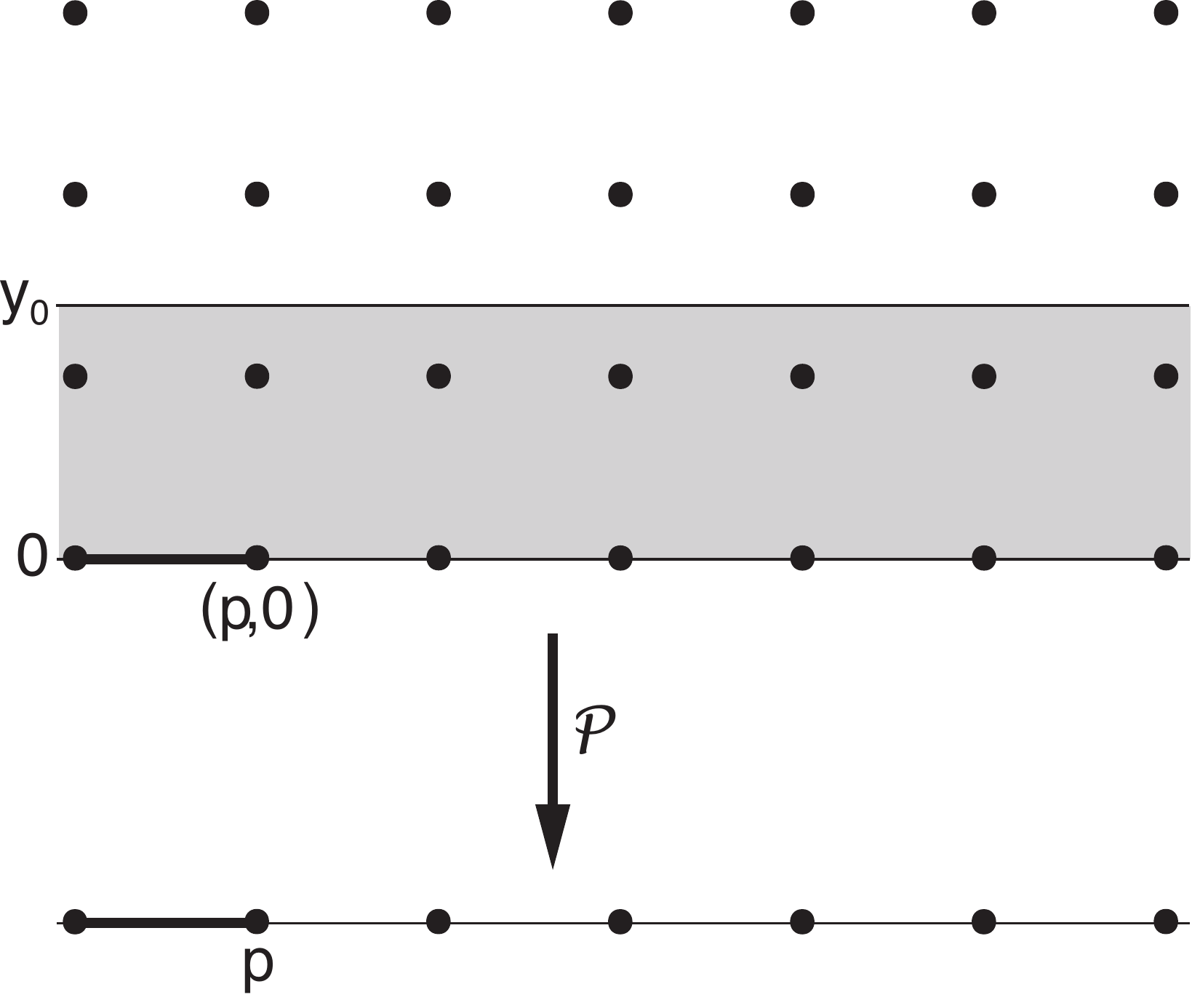}
\caption{Projections of a lattice: If $(\pp,0) \in \LL$ (Case~\ref{Latt1} of
Proposition~\ref{propLatt}) then both the projection
$\proj(\LL\cap\faixa)$ and the restriction to the line $y=0$
have period $\pp$.
}\label{figura-periodos-1}
\end{figure}

\begin{figure}[h]
\includegraphics[scale=0.3]{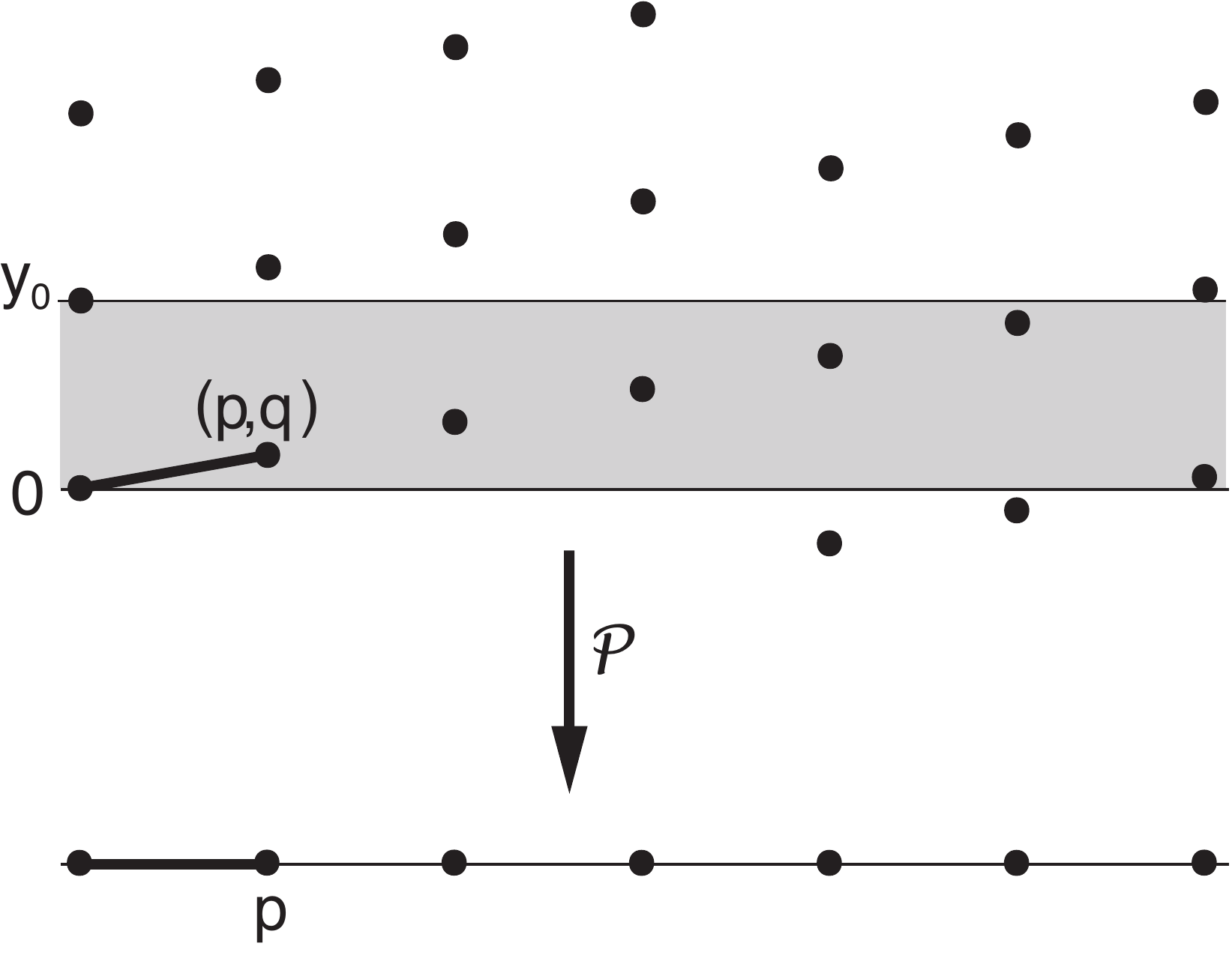}
\caption{Projections of a lattice:
$(0,y_0) \in \LL$ ensures period $\pp$ for the projection width $y_0$  (Case~\ref{Latt2} of
Proposition~\ref{propLatt})
even when the restriction to $y=0$ is
nonperiodic.}\label{figura-periodos-3}
\end{figure}

\begin{figure}[h]
\includegraphics[scale=0.3]{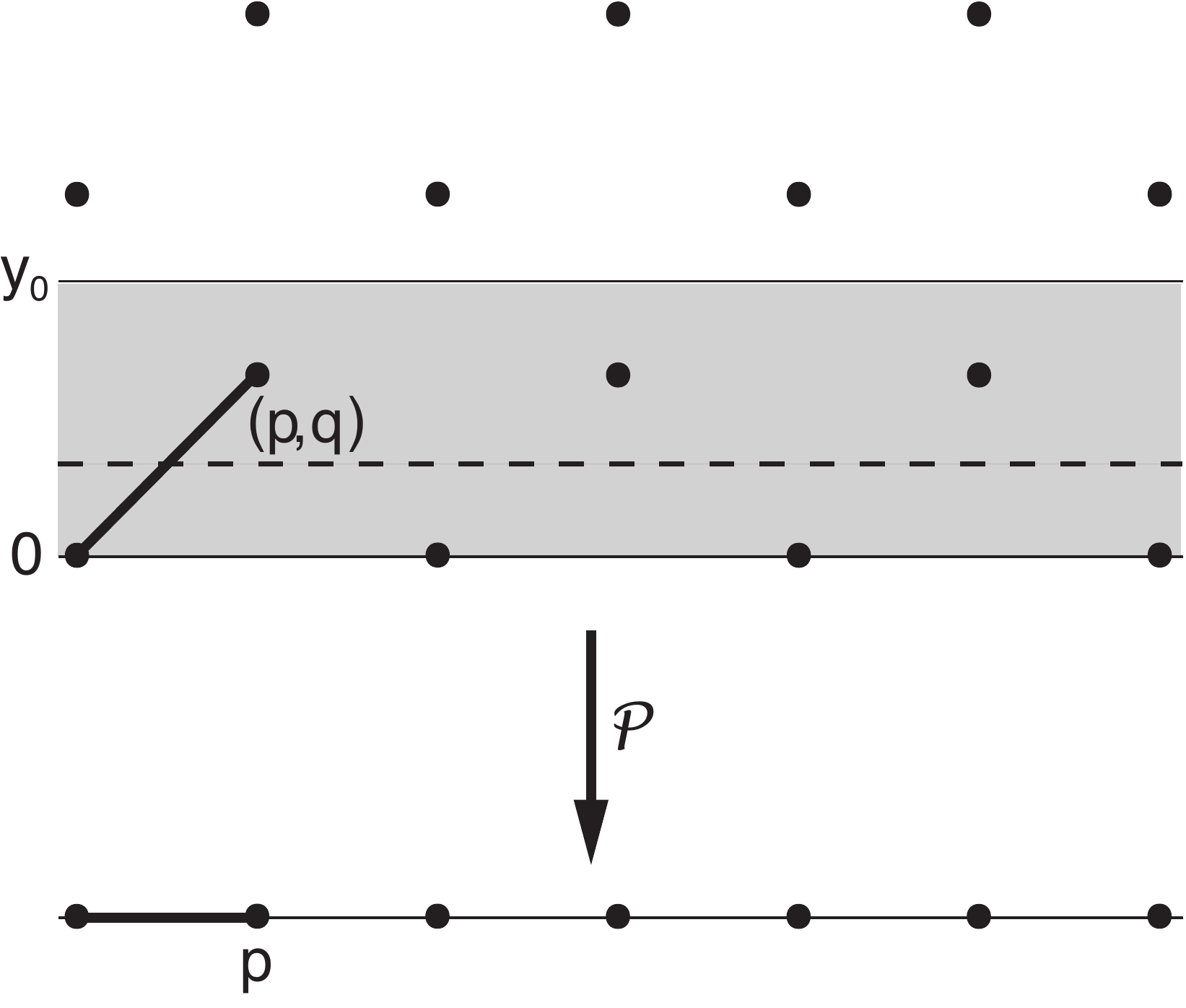}
\caption{Projections of a lattice:
The glide reflection on the dashed line (Case~\ref{Latt3} of
Proposition~\ref{propLatt}) acts
as a translation by $\pp$, after the projection for sufficiently large
 $y_0$.
The restriction to
$y=0$ and  projections of narrower strips have period
$2\pp$.}\label{figura-periodos-2} 
\end{figure}

\begin{figure}[h]
\includegraphics[scale=0.3]{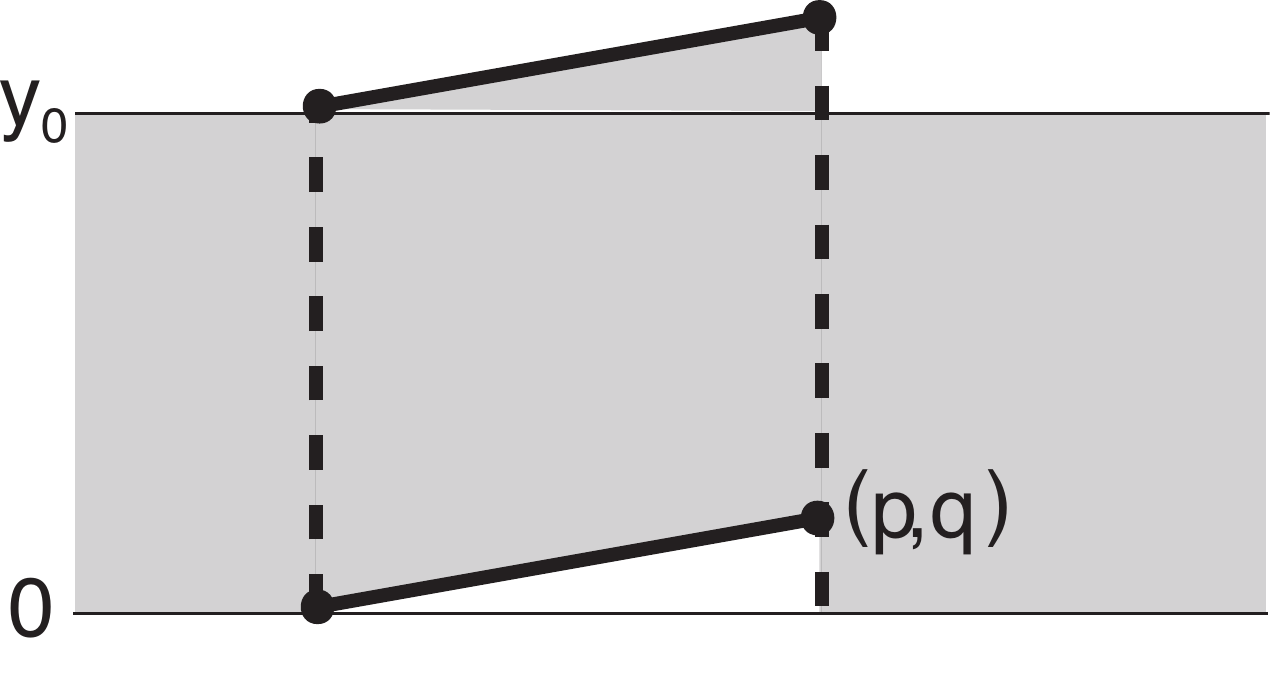}
\caption{If $(0,y_0) \in \LL$ then projecting a strip of
width $y_0$ is equivalent to the projection of the
cell defined by  $(0,y_0)$ and
$(\pp,a)$.}\label{figura-periodos-4}
\end{figure}

 \section{Patterns}\label{SecPatterns}
 We identify a  {\em pattern} to the level set of a function
$f:\RR^{n+1} \seta \RR,$
 periodic along $n+1$ independent directions.

\subsection{$\Gamma$ acting on functions}

The action of $\Gamma$ in $\RR^{n+1}$ induces the scalar
action: $(\gamma \cdot f)(x,y)=f(\gamma^{-1} \cdot (x,y))$ 
for $\gamma \in \Gamma$ and $(x,y) \in  \RR^{n+1}$, see
Melbourne~\cite[section 2.1]{Melbourne}. A function $f$ is 
$\Gamma${\sl -invariant} if
$(\gamma \cdot f)(x,y)=f(x,y)$, for all $\gamma \in
\Gamma$ and all $(x,y)
\in  \RR^{n+1}.$ 

\subsection{function spaces}
We will work in $X_\Gamma$,  the vector space
$$
X_\Gamma=\{f:\RR^{n+1}\seta\RR^n  \quad \Gamma\mbox{-invariant, of class } C^1  \}
$$
where $\Gamma$ is a
$(n+1)$-dimensional crystallographic group with lattice $\LL$ and point group $\ptg$. 
Since $\LL$ is a subgroup of $\Gamma$ then any $f\in X_\Gamma$ is
$\LL$-invariant. 
Thinking of $\LL$ as a subset of  $\RR^{n+1}$ this means that any $f\in X_\Gamma$ is 
$\LL$-{\em periodic}. 
Thus $X_\Gamma$ is  the
generalisation, to any dimension, of functions on the plane
whose level curves form a periodic tiling.

Consider the {\em waves}
$\omega_k(x,y)=\ee ^{2\pi i
\langle k,(x,y)\rangle},$ where $k \in \RR^{n+1}$
and $\langle \cdot , \cdot\rangle$ is the usual inner product in
$\RR^{n+1}$. 
The set of all $k \in \RR^{n+1}$ such that
$\omega_k$ is a
$\LL$-periodic function is
the dual lattice of $\LL$.
 If the  point group $\ptg$ of $\Gamma$ is non-trivial, then the waves $\omega_k(x,y)$, $k\in\Ld$ are not $\Gamma$-invariant, but a Hilbert basis for $X_\Gamma$ may be obtained from them by taking   $I_k(x,y) =\sum_{\delta \in \ptg} 
\omega_{\delta k}(x,y) \omega_{\delta
k}(-v_\delta)$ where $\left(v_\delta,\delta\right)\in\Gamma$. 
Note that this expression does not depend on the choice of $v_\delta$.
The functions $Re(I_k), Im(I_k)$ lie in $X_\Gamma$ for all $k\in \Ld$, 
and every $\Gamma$-invariant function of class $C^1$ has a
Fourier expansion in terms of the functions $I_k$ with  the  Fourier series  of $f$ converging absolutely and uniformly to $f$.

\subsection{the projection operator}

For $y_0>0$, consider the restriction of $f$ to 
the region between the {\em hyperplanes} $y=0$ and $y=y_0$.
The {\em projection operator} $\Pi_{y_0}$ integrates this
restriction of $f$ along the width $y_0$, yielding a
new function with domain $\RR^n$:

$$\Pi_{y_0}(f)(x)=\int_0^{y_0}f(x,y)dy.$$

The region between $y=0$ and $y=y_0$ is called the {\em
projected band} or the {\em projection band}, and $y_0$ is
called the {\em width of projection} or the {\em width of the
projected band}.

Note that  since the Fourier series of $f$ converges  absolutely and uniformly to $f$,
 it follows that the integral  in the projection of $f$ commutes with the summation in its Fourier series.

\subsection{symmetries of projected functions}

If $f \in X_\Gamma$ then the {\em projected function}
 $\Pi_{y_0}(f)$ may be invariant under the
action of some elements of the group $\Euc(n)  \cong
\RR^{n} \ltimes\Ort(n)$. Using a notation similar to the
$(n+1)$-dimensional case, 
$(v_\alpha, \alpha) \in \Euc(n)$ is a symmetry of 
$\Pi_{y_0}(f)$ if
$$(v_\alpha, \alpha) \cdot \Pi_{y_0}(f)(x)
= \Pi_{y_0}(f)(\alpha^{-1}x - v_\alpha)
=\Pi_{y_0}(f)(x) \quad \forall x \in \RR^n.$$

\subsection{restriction\label{seccao-notacao-restricao}}

Let $\Phi_r$ be the operator that restricts the functions to
the hyperplane $y=r$,
$$\Phi_r(f)(x)=f(x,r).$$
The functions $\Phi_r(f)$ may be
invariant under the action of some elements of the group
$\Euc(n)  \cong
\RR^{n} \ltimes\Ort(n)$, as discussed above for the projected
functions.

\subsection{ description of the symmetries of projected functions}

The following result from \cite{PinhoLabouriau} shows how to obtain symmetries  
of the projected functions
$\Pi_{y_0}(f)\in\Pi_{y_0}(X_\Gamma)$ from the symmetry of the functions $f\in X_\Gamma$.

\begin{teorema}[\cite{PinhoLabouriau}]\label{teorema-projeccao}

All functions in $\Pi_{y_0}(X_\Gamma)$ are
invariant under the action of $(v_\alpha,\alpha) \in 
\RR^n \ltimes \Ort(n)$ if and only
if one of the following conditions holds:

\begin{enumerate}
\renewcommand{\theenumi}{\alph{enumi}}

\item\label{teorema-1} $\left((v_\alpha,0),\alpha_+ \right)
\in \Gamma$,

\item\label{teorema-2} $\left((v_\alpha,y_0),\alpha_- \right)
\in \Gamma$,

\item\label{teorema-3} 
$(0,y_0) \in \LL$ and 
$\left((v_\alpha,y_1),\alpha_+ \right) \in \Gamma$, for some
$y_1 \in \RR$ .

\item\label{teorema-4}
 $(0,y_0) \in \LL$ and 
$\left((v_\alpha,y_1),\alpha_- \right) \in \Gamma$, for some
$y_1 \in \RR$ .

\end{enumerate}

\end{teorema}

A similar result holds for the restriction:
\begin{teorema}[\cite{PinhoLabouriau}]\label{teorema-restricao}

All functions in $\Phi_r(X_\Gamma)$ are
invariant under the action of $(v_\alpha,\alpha) \in 
\RR^n \ltimes \Ort(n)$ if and only
if one of the following conditions holds:

\begin{enumerate}
\renewcommand{\theenumi}{\alph{enumi}}

\item\label{teorema-restricao-1} 
$\left((v_\alpha,0),\alpha_+
\right) \in \Gamma$,

\item\label{teorema-restricao-2}
$\left((v_\alpha,2r),\alpha_- \right) \in \Gamma$.

\end{enumerate}

\end{teorema}

 \section{The symmetry group of projected functions}\label{seccao-periodos-funcoes}
 
 Let $\Gamma\subset\Euc(n+1)$ be a crystallographic group with lattice $\LL$ and point group $\ptg$.
 In this section we obtain a complete description of the group
 $\Gtil\le\Euc(n)$ of symmetries shared by all the projected patterns in $\Pi_{y_0}(X_\Gamma)$, by describing
  its translation subgroup $\Ltily$ 
  and its orthogonal component.  
 In particular, we 
describe conditions ensuring that the projections of
$\LL$-periodic functions are still periodic functions with $n$ linearly independent periods, i.e. conditions ensuring that $\Ltily$ is a lattice with $n$ generators.

We start by showing that the translation subgroup $\Ltily$
of common periods of the projected functions in $\Pi_{y_0}(X_\Gamma)$,
regarded as a $\ZZ$-module, has  at most $n$ generators.
We also characterise the situations when $\Ltily$ has 
 exactly $n$ generators and thus  is a lattice. 
This means that the symmetry group of the projected functions is a crystallographic group.

Let $\LL_\sigma=v_\sigma+\LL$ if 
$\left(v_\sigma,\sigma\right)\in\Gamma$ and $\LL_\sigma=\emptyset$ if $\sigma\not\in\ptg$.
In all cases $\Ltily$ is a subset of $\{\proj(\LL),\proj(\LL_\sigma)\}_{\ZZ}$,
by Theorem~\ref{teorema-projeccao}.

 \begin{teorema}\label{ThLtilModulo}
 Let $\Gamma\subset\Euc(n+1)$ be a crystallographic group with lattice $\LL$ and point group $\ptg$.
Then the  group $\Ltily$  of common periods of $\Pi_{y_0}(X_\Gamma)$, regarded as a $\ZZ$-module, has  at most $n$ generators.
 If either $(0,y_0)\in\LL$ or 
  $\sigma\in\ptg$,  then $\Ltily$ is a lattice.
 \end{teorema}
 
  \begin{proof}
 We use Theorem~\ref{teorema-projeccao} in the case $\alpha=I\!d_n$.
 If  $(0,y_0)\not\in\LL$ then  $\pp\in\Ltily$ if and only if one of the conditions  (\ref{teorema-1}) or (\ref{teorema-2}) of Theorem~\ref{teorema-projeccao} holds for $v_\alpha=\pp$. 
Condition  (\ref{teorema-1}) means 
$\pp\in\restr_0(\LL)$. 
Condition (\ref{teorema-2}) means 
$(\pp,y_0) \in \LL_\sigma$, which is equivalent to 
$\pp\in\restr_{y_0}(\LL_\sigma)$.
Thus  
\begin{equation}\label{LsigmaSemy0}
(0,y_0)\not\in\LL
\qquad \Longrightarrow\qquad
\Ltily=\restr_0(\LL)\cup \restr_{y_0}(\LL_\sigma)
\end{equation}
This case is treated in Lemma~\ref{lemay0foraL} below.

If  $(0,y_0)\in\LL$ and $\pp\in\Ltily$ 
then any one of the conditions  of Theorem~\ref{teorema-projeccao} may hold. 
The first two imply 
$\pp\in\restr_0(\LL)\cup \restr_{y_0}(\LL_\sigma)\subset\proj(\LL)\cup \proj(\LL_\sigma)$ 
as seen above.
The last two imply that  $(\pp,y_1)\in\LL\cup\LL_\sigma$ for some $y_1\in\RR$
 and therefore $\Ltily\subset\proj(\LL)\cup \proj(\LL_\sigma)$.
It follows immediately from Theorem~\ref{teorema-projeccao} that
\begin{equation}\label{LsigmaComy0}
(0,y_0)\in\LL
\qquad \Longrightarrow\qquad
\Ltily=\proj(\LL)\cup \proj(\LL_\sigma).
 \end{equation}
 It remains to show that this implies that $\Ltily$ is a lattice in $\RR^n$,
 we treat this case in Lemma~\ref{lemay0emL}.
 \end{proof}

\begin{lema}\label{lemay0emL}
Let $\Gamma\subset\Euc(n+1)$ be a crystallographic group with lattice $\LL$.
If $(0,y_0)\in\LL$ then 
the group $\Ltily$  of common periods of all functions in $\Pi_{y_0}(X_\Gamma)$
 is the lattice $\Ltily=\proj(\LL)\cup \proj(\LL_\sigma)$.
 \end{lema}

\begin{proof}
Let $a=y_0/m$ where $m$ is the largest integer such that  $(0,y_0/m)\in\LL$.
We may write $\LL=\{(0,a),l_1,\ldots,l_n \}_\ZZ$ where the $l_i$, $i=1,\ldots,n$ are
linearly independent over $\RR$,
therefore $\proj(l_i)$,  are linearly independent  generators for $\proj(\LL)$.
If $\sigma$ is not in the point group $\ptg$ of $\Gamma$, then by (\ref{LsigmaComy0})
we have that $\Ltily=\proj(\LL)$ has $n$ generators.

If $\left((v_1,y_1),\sigma\right)\in\Gamma$ and  if $v_1=0\pmod{\LL}$ then $\proj(\LL_\sigma)\subset \proj(\LL)$ and hence $\Ltily=\proj(\LL)$ has $n$ generators. Otherwise,
\begin{equation}\label{sigma2a}
\left((v_1,y_1),\sigma\right)\cdot\left((v_1,y_1),\sigma\right)=
\left((2v_1,0),I\!d_{n+1}\right)\in\Gamma
\quad\Longrightarrow\quad
(2v_1,0)\in\LL
\end{equation}
and therefore $2v_1=2\proj(v_1,y_1)\in\proj(\LL)$.
To see that $\Ltily$ is a lattice  in this case,
note that $\left((v_1,y_1),\sigma\right)\in\Gamma$ implies that there exists $v_2=\sum_{i=1}^n s_i \proj(l_i)$ 
with $0\le s_i<1$ and $\left((v_2,y_2),\sigma\right)\in\Gamma$ for some $y_2\in\RR$,
and $2v_2\in\proj(\LL)$ by \eqref{sigma2a}.
Hence, each $s_i$ is either 0 of 1/2 and not all of them are zero.
Without loss of generality, suppose
 $s_1=1/2$, then    $v_2=\frac{1}{2}\proj( l_1)+ \sum_{i=2}^n s_i\proj( l_i)$ and thus
$\proj(l_1)=2v_2- \sum_{i=2}^n 2s_i \proj(l_i)$ with $2s_i\in\ZZ$, hence 
$\Ltily=\proj(\LL)\cup \proj(\LL_\sigma)=\left\{ v_2,\proj(l_2),\ldots, \proj(l_n)\right\}_\ZZ$ and the result follows.
\end{proof}

\begin{lema}\label{lemay0foraL}
Let $\Gamma\subset\Euc(n+1)$ be a crystallographic group with lattice $\LL$ and point group $\ptg$
and suppose $(0,y_0)\not\in\LL$.
If  $\left((v_1,y_0),\sigma \right) \in \Gamma$
then  the group $\Ltily$  of common periods of all functions in $\Pi_{y_0}(X_\Gamma)$ is the lattice $\Ltily=\{\restr_0(\LL),v_1\}_{\ZZ}$.
If $\left((v_1,y_0),\sigma \right) \not\in \Gamma$
then $\Ltily=\restr_0(\LL)$. 
If $\sigma$ is in the holohedry of $\LL$
then $\Ltily$ is a lattice, otherwise $\Ltily$
may have less than $n$  linearly independent  generators.
\end{lema}

\begin{proof}
First we show that if $\sigma\LL=\LL$ then $\restr_0(\LL)$ is a lattice.
To see this, note that if $(x,y)\in\LL$, then 
$(x,y)-\sigma(x,y)=(0,2y)\in\LL$.
Thus, if $y\ne 0$, we 
may write
$$
\LL=\left\{(0,b),(a_1,0),\ldots(a_k,0),(a_{k+1},b/2),\ldots,(a_n,b/2) \right\}_\ZZ
$$
for some $k$, $1\le k\le n$, with the $a_i$ linearly independent over $\RR$.

If $k=n$ the claim is proved.
If $k=n-1$ then 
$$
\LL=\left\{(0,b),(a_1,0),\ldots(a_{n-1},0),(a_{n},b/2)\right\}_\ZZ
$$
and $(2a_n,0)\in\LL$ hence $\restr_0(\LL)=\left\{a_1,\ldots,a_{n-1},2a_n\right\}$ is a lattice.
Otherwise note that 
for $i=1,\ldots,n-k$, and taking $a_{n+1}=a_{k+1}$ we have
$$(a_{k+i},b/2)+(a_{k+i+1},b/2)-(0,b)=(a_{k+i}+a_{k+i+1},0)\in\LL \pmod{n-k}$$
and thus
$$\restr_0(\LL)=\left\{a_1,\ldots,a_k,a_{k+1}+a_{k+2},\ldots,a_n+a_{k+1} \right\}_\ZZ$$
proving the claim that  $\restr_0(\LL)$ is a lattice.

If  $\left((v_1,y_0),\sigma \right) \in \Gamma$ then 
$$
\restr_{y_0}(\LL_\sigma)=
\left\{ x=v_1+a: \ (a,b)\in\LL \mbox{ and } y_0+b=y_0 \right\}_\ZZ
= v_1+\restr_0(\LL) .
$$
 If $v_1\in\restr_0(\LL)$ then, since $(0,y_0)\not\in\LL$,  it follows  from  (\ref{LsigmaSemy0}) that
 $\Ltily=\{\restr_0(\LL),v_1\}_{\ZZ}=\restr_0(\LL)$.
Otherwise $2v_1\in\restr_0(\LL)$ by~(\ref{sigma2a}) 
and $\Ltily$ is a lattice 
  as claimed.

If $(0,y_0)\not\in\LL$ and $\left((v_1,y_0),\sigma \right) \not\in \Gamma$ then 
$\Ltily=\restr_0(\LL)$, by  (\ref{LsigmaSemy0}).
If $\sigma$ is in the holohedry of $\LL$ then $\Ltily=\restr_0(\LL)$ is a lattice, 
otherwise it may have fewer than $n$ generators.
\end{proof}

In the next results we use the symbol
$(v_\delta,\delta)$ for any element
of the coset $\LL
\cdot (v,\delta)=\{(l+v,\delta) : l\in \LL\}$ for any $v\in 
\RR^{n+1}$ such that $(v, \delta) \in \Gamma$, {\sl i.e.}, $v_\delta$ is the
{\em non-orthogonal component} of $(v,\delta)\in \Gamma$
defined up to elements of
$\LL$. 

A direct application of Theorem~\ref{teorema-projeccao} yields:

 \begin{proposicao}\label{ThJtil}
 Let $\Gamma\subset\Euc(n+1)$ be a crystallographic group with lattice $\LL$ and point group $\ptg$ and let $\ptgtil$ be the point group of the symmetry group of $\Pi_{y_0}(X_\Gamma)$.
Then  $\ptgtil$ is a subgroup of 
$$\ptg_0=\left\{\alpha: \mbox{either }
\alpha_+\in\ptg \mbox{ or } \alpha_-\in\ptg\right\}\le\Ort(n),
$$
satisfying:
 \begin{enumerate}
 \item\label{JtilComy0}
 if $(0,y_0)\in\LL$ then  $\ptgtil=\ptg_0$;
 \item\label{JtilSemy0}
  if $(0,y_0)\not\in\LL$ then  \\
  $\ptgtil=\{\alpha\in\ptg_0 : \mbox{either }v_{\alpha_{+}}=(v,0) \mbox{ or }
v_{\alpha_{-}}=(v,y_0)  \mbox{ for some } v\in\RR^n\}$.
 \end{enumerate}
 Moreover, for all $\alpha\in\ptgtil$ we have either
 $v_\alpha=\proj( v_{\alpha_{+}})$ or 
 $v_\alpha=\proj( v_{\alpha_{-}})$.
 \end{proposicao}

  A similar and simpler application of Theorem~\ref{teorema-restricao} yields a result for  the restriction:
 
  \begin{proposicao}\label{PropRestrLsigma}
 Let $\Gamma\subset\Euc(n+1)$ be a crystallographic group with lattice $\LL$ and point group $\ptg$.
 Let $\widehat{\Gamma}$ be  the group of symmetries shared by all functions in $\Phi_r(X_\Gamma)$.
Then  the subgroup of translations $\widehat\LL$ of  $\widehat{\Gamma}$ is
 $$
 \widehat\LL=\restr_0(\LL)\cup \restr_{2r}(\LL_\sigma),
 $$
 and the point group ${{\rm\bf \widehat J}}$ of $\widehat{\Gamma}$ is
$${{\rm\bf \widehat J}}=\{\alpha\in\ptg_0 :  \mbox{either } v_{\alpha_{+}}=(v,0) \mbox{ or }
v_{\alpha_{-}}=(v,2r)  \mbox{ for some } v\in\RR^n\}$$
and for all $\alpha\in{{\rm\bf \widehat J}}$ we have either
 $v_\alpha=\proj( v_{\alpha_{+}})$ or 
 $v_\alpha=\proj( v_{\alpha_{-}})$.
 \end{proposicao}

\section{Finding the groups that project into a given lattice}\label{secProjectGeneral}
From now on we specialise to the case $n=2$.
 Suppose that, for some value of the width $y_0$, 
 the  group $\Ltily$  of common periods of $\Pi_{y_0}(X_\Gamma)$ is a given lattice $\MM\subset\plano$.
We list all the possible lattices $\LL$ that lead to this result for some $y_0$ 
and describe what happens for other  values of $y_0$ in each case. 

We  say that a
 lattice $\LL_1$ is {\em rationally compatible} with another lattice $\LL_2$ if there exists $m\ne 0 \in\ZZ$ such that $m\LL_1\subset\LL_2$.
One common situation in the context above is that either $\Ltily$  is rationally compatible with $\MM$ or $\MM$  is rationally compatible with $\Ltily$.

The lattices $\LL$
are described up to symmetries of the form $\alpha_+$, where $\alpha$ is an element of the holohedry of $\MM$.
Thus, we address
the question of how the set of periods $\Ltily$,  and more generally the group of symmetries $\Gtil_{y_0}$ of the projected functions, changes
 with the width $y_0$ of the projection band, knowing that one of the projected patterns has the periods of 
 $\MM$.
We start with a procedure for a general lattice $\MM=\left\{\oa_0,\ob_0\right\}_\ZZ$ and in the next section we specialise to a hexagonal lattice.

We use the structure of the proof of Theorem~\ref{ThLtilModulo}, to obtain special forms for the generators
$\oa$, $\ob$, $\oc$ of $\LL$ according to the following cases:

 \noindent {\bf case 1}.
$\LL\cap\left\{(0,0,z): z\in\RR\right\}=\left\{(0,0,0)\right\}$.
This condition is not compatible with $\sigma\LL\subset\LL$, as shown in the proof of Lemma~\ref{lemay0foraL}.
Hence, by \eqref{LsigmaSemy0}, for all $y_0$ we have $\Ltily=\restr_0(\LL)=\MM$ and
$\LL=\left\{\left(\oa_0,0\right),\left(\ob_0,0\right), \left(\oc_0,c_3\right)\right\}_\ZZ$ for some $\oc_0\in\plano$, $c_3\in\RR$ with $n\oc_0\not\in\MM$ for all $n\in\ZZ$ and $c_3\ne0$.

\bigbreak
For all the other cases we need to establish some notation.
Let $\oc=(0,0,c)\in\LL$, minimal in its direction.
Then for $y_0=nc$, $n\in\ZZ$ we have $\Ltily=\proj(\LL)\cup \proj(\LL_\sigma)$ by \eqref{LsigmaComy0}.
Other values of $y_0$ yield $\Ltily=\restr_0(\LL)\cup \restr_{y_0}(\LL_\sigma)$ by \eqref{LsigmaSemy0}.
The cases below correspond to the different situations with respect to $\sigma$.

We may write the generators of $\LL$ in the form $\oa=\left(\na ,a_3\right)$, $\ob=\left(\nb ,b_3\right)$ and  $\oc=\left(0,0,c\right)$ with $\na ,\nb \in\plano$, $c>0$ and $a_3, b_3\in\left[0,c\right)$.
Then $\sigma\LL\subset\LL$ if and only if $a_3, b_3\in\left\{0,c/2\right\}$.

The number of  generators of $\restr_0(\LL)$ depends on metric properties of the generators of $\LL$ as follows.
Let $D(\LL)$ be the $\ZZ$-module
\begin{equation}\label{defD}
D(\LL)=\{(m,n)\in \ZZ^2: ma_3+nb_3=0 \pmod{c}\}.
\end{equation}
Then 
$$
\restr_0(\LL)=\left\{ m\na +n\nb, \  (m,n)\in D(\LL)\right\} .
$$
Generically, $D(\LL)=\left\{(0,0)\right\}$, but in special situations we may have that the $\ZZ$-module   $D(\LL)$ is either
generated by one non-zero element or by two  linearly independent elements of $\ZZ^2$.
\bigbreak

 \noindent {\bf case 2}.
$\sigma\LL\not\subset\LL$, hence  either $a_3\ne 0$ or $b_3\ne 0$.
Then $\Ltily=\proj(\LL)$ for $y_0=nc$, $n\in\ZZ$.
For other values of $y_0$, we have $\Ltily=\restr_0(\LL)$ and
$\Gtil_{y_0}$ is a subperiodic group, except 
in the special case where $D(\LL)$ has two generators. There are two possibilities:

\noindent {\bf case 2.1} If $\proj(\LL)=\MM$
we may take $\oa=\left(\oa_0,a_3\right)$, $\ob=\left(\ob_0,b_3\right)$ and $\Ltily=\MM$ for $y_0= nc$, $n\in\ZZ$.
Then $\Ltily$ is a proper subset of $\MM$ for other values of $y_0$,.

\noindent {\bf case 2.2} If   $D(\LL)$ has two independent generators, we may also have $\restr_0(\LL)=\MM$.
Then  $\Ltily=\MM$ for $y_0\ne nc$, and $\MM$  is a proper subset of $\Ltily$ for $y_0=nc$.

\bigbreak

 From now on we assume $\sigma\LL\subset\LL$, hence we have $a_3,b_3\in\left\{0,c/2\right\}$, not both equal to $c/2$, as in the proof of Lemma~\ref{lemay0foraL}.
Then $D(\LL)$ always has two independent generators, according to the following table
\begin{center}
\begin{tabular}{c|l|l}
&$ b_3=0$&$ b_3=c/2$\\
\hline
$a_3=0$&$D(\LL)=\ZZ^2$&$D(\LL)=\left\{(1,0),(0,2)\right\}_\ZZ$\\
\hline
$a_3=c/2$&$D(\LL)=\left\{(2,0),(0,1)\right\}_\ZZ$& ---
\end{tabular}
\end{center}
and  $\restr_0(\LL)$ is given by:
\begin{center}
\begin{tabular}{c|l|l}
&$ b_3=0$&$ b_3=c/2$\\
\hline
$a_3=0$&$\restr_0(\LL)=\left\{\na ,\nb \right\}_\ZZ$&$\restr_0(\LL)=\left\{\na ,2\nb \right\}_\ZZ$\\
\hline
$a_3=c/2$&$\restr_0(\LL)=\left\{2\na ,\nb \right\}_\ZZ$& ---
\end{tabular}
\end{center}

\bigbreak

 \noindent {\bf case 3}.
If $\sigma\LL\subset\LL$ and $\LL_\sigma\subset\LL$ (this holds in particular, if $\sigma\not\in\ptg$),
then for   $y_0=nc$, $n\in\ZZ$ we have $\Ltily=\proj(\LL)$.
Other values of $y_0$ yield $\Ltily=\restr_0(\LL)$.
There are three possibilities:

\noindent {\bf case 3.1} If $a_3=b_3=0$ then $D(\LL)=\ZZ^2$, $\restr_0(\LL)=\proj(\LL)=\MM$ for all $y_0$.

 \noindent {\bf case 3.2} If  either $a_3\ne 0$ or $b_3\ne 0$ and if
$\Ltily=\proj(\LL)=\MM$ for  $y_0=nc$, $n\in\ZZ$, then $\na =\oa_0$, $\nb =\ob_0$.
For other values of $y_0$, we have $\Ltily=\restr_0(\LL)$ as in the table above.
 
 \noindent {\bf case 3.3}  If  either $a_3\ne 0$ or $b_3\ne 0$ and if
 $\Ltily=\restr_0(\LL)=\MM$ for  $y_0\ne nc$, $n\in\ZZ$, then
either the generators of $\LL$ are $\left(\oa_0,0\right)$ and $\frac{1}{2}\left(\ob_0,c\right)$ and hence for $y_0=nc$, $n\in\ZZ$, we have
$\Ltily=\proj(\LL)=\left\{\oa_0,\frac{1}{2}\ob_0\right\}_\ZZ$,
or the generators are $\frac{1}{2}\left(\oa_0,c\right)$ and $\left(\ob_0,0\right)$ and hence for $y_0=nc$, $n\in\ZZ$, we have
$\Ltily=\proj(\LL)=\left\{\frac{1}{2}\oa_0,\ob_0\right\}_\ZZ$.

In both cases 3.2 and 3.3, for all values of $y_0$ we obtain that $\Ltily$ is a lattice and that 
 $\MM$ is rationally compatible with $\Ltily$ even when $(0,y_0)\not \in\LL$ and $\sigma\not\in\ptg$.
\bigbreak

 From now on we assume 
 $\left(v_\sigma,\sigma\right)\in\Gamma$ where
$v_\sigma=\left(v_1,y_1\right)\not\in\LL$ with $y_1\in[0,c)$.
 By \eqref{sigma2a} it follows that  $2v_1\in\restr_0(\LL)$.
We obtain the  cases below, depending on the values of $a_3$ and $b_3$.
In all cases, if $\Ltily=\MM$ for some $y_0$, then  $\Ltily$ is rationally compatible with $\MM$ for every $y_0$.

\bigbreak
 
\noindent {\bf case 4}. If $a_3=b_3=0$ and $\left(v_\sigma,\sigma\right)\in\Gamma$ then 
 $v_1\in\left\{0,\frac{1}{2}\na ,\frac{1}{2}\nb ,\frac{1}{2}\left(\na +\nb \right)\right\}$, by the arguments in the proof of Lemma~\ref{lemay0emL}.

\noindent {\bf case 4.1}.
$\LL=\left\{\left(\oa_0,0\right),\left(\ob_00\right),  \left(0,0,c\right)\right\}_\ZZ$.\\
 $v_1=0$, then  $\Ltily=\MM$  for all $y_0$.\\
 $v_1=\frac{1}{2}\oa_0$ then $\Ltily=\left\{\frac{1}{2}\oa_0,\ob_0\right\}_\ZZ$
 for $y_0=nc$ and $y_0=nc+y_1$, $n\in\ZZ$; and $\Ltily=\MM$ for other $y_0$.\\
 $v_1=\frac{1}{2}\ob_0$ then $\Ltily=\left\{\\oa_0,\frac{1}{2}\ob_0\right\}_\ZZ$
 for $y_0=nc$ and $y_0=nc+y_1$, $n\in\ZZ$; and $\Ltily=\MM$ for other $y_0$.\\
$v_1=\frac{1}{2}\left(\oa_0+\ob_0\right)$ then $\Ltily=\left\{\oa_0,\frac{1}{2}\left(\oa_0+\ob_0\right)\right\}_\ZZ$
 for $y_0=nc$ and $y_0=nc+y_1$, $n\in\ZZ$; and $\Ltily=\MM$ for other $y_0$.

\noindent {\bf case 4.2}. $v_1=\oa_0=\frac{1}{2}\na$, and 
$\LL=\left\{\left(2\oa_0,0\right),\left(\ob_0,0\right),  \left(0,0,c\right)\right\}_\ZZ$. 
Then $\Ltily=\MM$  for $y_0=nc$ and $y_0=nc+y_1$, $n\in\ZZ$; and $\Ltily=\left\{ 2\oa_0,\ob_0\right\}$ for other $y_0$.

\noindent {\bf case 4.3}. $v_1=\ob_0=\frac{1}{2}\nb$ or $v_1=\ob_0=\frac{1}{2}\left(\na+\nb\right)$, and 
$\LL=\left\{\left(\oa_0,0\right),\left(2\ob_0,0\right),  \left(0,0,c\right)\right\}_\ZZ$.
Then $\Ltily=\MM$  for $y_0=nc$ and $y_0=nc+y_1$, $n\in\ZZ$; and $\Ltily=\left\{ \oa_0,2\ob_0\right\}$ for other $y_0$.

\bigbreak

 \noindent {\bf case 5}.
 If $a_3=c/2$ and $b_3=0$ and $\left(v_\sigma,\sigma\right)\in\Gamma$ then 
$v_1\in\left\{0,\na ,\frac{1}{2}\nb ,\na +\frac{1}{2}\nb \right\}$, by the arguments in the proof of Lemma~\ref{lemay0emL}. 
There are several possibilities for $\Ltily$:

 \noindent 
For $y_0=nc$, $n\in\ZZ$, $\Ltily=\proj(\LL)+\proj(\LL_\sigma)=\proj(\LL)+v_1=\left\{\na,\nb\right\}_\ZZ\cup\left( \left\{\na,\nb\right\}_\ZZ+v_1\right)$.

 \noindent 
For $y_0= nc+y_1$, $n\in\ZZ$,  then $\Ltily=\restr_0(\LL)\cup \restr_{y_0}(\LL_\sigma)=\{2\na ,\nb\}_\ZZ\cup \left(\{2\na ,\nb\}_\ZZ+v_1\right)$.

 \noindent 
For $y_0= nc+\frac{1}{2}c+y_1$, $n\in\ZZ$,  then $\Ltily=\restr_0(\LL)\cup \restr_{y_0}(\LL_\sigma)=\{2\na ,\nb\}_\ZZ\cup \left(\{\na ,\nb\}_\ZZ+v_1\right)$.

 \noindent 
For all other values of $y_0$ we get $\Ltily=\restr_0(\LL)=\left\{2\na,\nb\right\}_\ZZ$.

\noindent
In all cases, if $\Ltily=\MM$ for some $y_0$, then  $\Ltily$ is rationally compatible with $\MM$ for every $y_0$.

 If $a_3=0$ and $b_3=c/2$ then the results are obtained from this case by interchanging $a$ and $b$.
The possibilities  with $a_3=c/2$ and $b_3=0$ are:

 \noindent {\bf case 5.1}.  
 $\na=\oa_0$, $\nb=\ob_0$ with $v_1\in \{0, \oa_0\}$.
 
 For $v_1=0$, $\Ltily=\MM$ for $y_0= nc$ and $y_0= nc+\frac{1}{2}c+y_1$, $n\in\ZZ$; and $\Ltily=\left\{2\oa_0,\ob_0\right\}_\ZZ$ for all other values of $y_0$. 
 
 For $v_1=\oa_0$, $\Ltily=\MM$ for $y_0= nc$ and $y_0= nc+y_1$, $n\in\ZZ$; and $\Ltily=\left\{2\oa_0,\ob_0\right\}_\ZZ$ for all other values of $y_0$. 
  
 \noindent {\bf case 5.2}. 
 $\na=\oa_0$, $\nb=2\ob_0$ with  $v_1\in\{\frac{1}{2}\ob_0,\oa_0+\frac{1}{2}\ob_0\}$, then $\Ltily=\MM$  for $y_0= nc$,  $n\in\ZZ$.
 
For  $v_1=\frac{1}{2}\ob_0$ then $\Ltily =\left\{2\oa_0,\frac{1}{2}\ob_0\right\}_\ZZ$ when $y_0= nc+y_1$, $n\in\ZZ$; and 
   $\Ltily =\left\{2\oa_0,\oa_0+\frac{1}{2}\ob_0\right\}_\ZZ$ for  $y_0= nc+\frac{1}{2}c+y_1$, $n\in\ZZ$.
 For all other values of $y_0$, we have
$\Ltily=\left\{2\oa_0,2\ob_0\right\}_\ZZ$. 
   
   For $v_1=\oa_0+\frac{1}{2}\ob_0$ then $\Ltily =\left\{2\oa_0,\oa_0+\frac{1}{2}\ob_0\right\}_\ZZ$ for $y_0= nc+y_1$, $n\in\ZZ$; and  $\Ltily =\left\{2\oa_0,\frac{1}{2}\ob_0\right\}_\ZZ$ for  $y_0= nc+\frac{1}{2}c+y_1$, $n\in\ZZ$.
   For all other values of $y_0$, we have
   $\Ltily =\left\{2\oa_0,2\ob_0\right\}_\ZZ$. 
 
\noindent {\bf case 5.3}.
$\na=\oa_0-\frac{1}{2}\ob_0$, $\nb=\ob_0$ with  $v_1\in\{\frac{1}{2}\ob_0,\oa_0\}$.

If $v_1=\frac{1}{2}\ob_0$ then $\Ltily=\MM$ for  $y_0= nc+\frac{1}{2}c+y_1$. Other values of $y_0$ yield
$\Ltily=\left\{\oa_0,\frac{1}{2}\ob_0\right\}_\ZZ$ for  $y_0= nc$, $n\in\ZZ$; and
$\Ltily =\left\{2\oa_0,\frac{1}{2}\ob_0\right\}_\ZZ$ for $y_0= nc+y_1$, $n\in\ZZ$.
All other values of $y_0$ correspond to  $\Ltily =\left\{2\oa_0,\ob_0\right\}_\ZZ$. 

If $v_1=\oa_0$ then $\Ltily=\MM$ for  $y_0= nc+y_1$. Other values of $y_0$ yield
$\Ltily=\left\{\oa_0,\frac{1}{2}\ob_0\right\}_\ZZ$ for  $y_0= nc$, $n\in\ZZ$; and
$\Ltily =\left\{2\oa_0,\frac{1}{2}\ob_0\right\}_\ZZ$ for $y_0= nc+\frac{1}{2}c+y_1$, $n\in\ZZ$.
All other values of $y_0$ correspond to  $\Ltily =\left\{2\oa_0,\ob_0\right\}_\ZZ$. 

\noindent {\bf case 5.4}.
$\na=\frac{1}{2}\oa_0$, $\nb=2\ob_0$,  $v_1\in\{\ob_0,\frac{1}{2}\oa_0+\ob_0\}$.

For $v_1=\ob_0$, then $\Ltily=\MM$  for $y_0= nc+y_1$.
Also, $\Ltily=\left\{\frac{1}{2}\oa_0,\ob_0\right\}_\ZZ$ for  $y_0= nc$; and
$\Ltily=\left\{\oa_0,\frac{1}{2}\oa_0+\ob_0\right\}_\ZZ$ for  $y_0= nc+\frac{1}{2}c+y_1$,  $n\in\ZZ$; and
$\Ltily=\left\{\oa_0,2\ob_0\right\}_\ZZ$ for  all other values of $y_0$. 

If $v_1=\frac{1}{2}\oa_0+\ob_0$, then $\Ltily=\MM$  for $y_0= nc+\frac{1}{2}c+y_1$,  $n\in\ZZ$.
In this case, $\Ltily=\left\{\frac{1}{2}\oa_0,\ob_0\right\}_\ZZ$ for  $y_0= nc$; and
$\Ltily=\left\{\oa_0,\frac{1}{2}\oa_0+\ob_0\right\}_\ZZ$ for  $y_0= nc+y_1$,  $n\in\ZZ$; and
$\Ltily=\left\{\oa_0,2\ob_0\right\}_\ZZ$ for  all other values of $y_0$.

\noindent {\bf case 5.5}.
$\na=\frac{1}{2}\oa_0$, $\nb=\ob_0$,  $v_1\in\{0,\frac{1}{2}\oa_0,\frac{1}{2}\ob_0,2\oa_0+\frac{1}{2}\ob_0\}$.

For $v_1=0$ then  $\Ltily=\left\{\frac{1}{2}\oa_0,\ob_0\right\}_\ZZ$ for  $y_0= nc$ or $y_0= nc+\frac{1}{2}c+y_1$; and $\Ltily=\MM$ for other values of $y_0$ including  $y_0= nc+y_1$,  $n\in\ZZ$.

For $v_1=\frac{1}{2}\oa_0$ then  $\Ltily=\left\{\frac{1}{2}\oa_0,\ob_0\right\}_\ZZ$ for  $y_0= nc$ or $y_0= nc+y_1$; and $\Ltily=\MM$  for other values of $y_0$ including   $y_0= nc+\frac{1}{2}c+y_1$,  $n\in\ZZ$.

For $v_1=\frac{1}{2}\ob_0$ then  $\Ltily=\left\{\frac{1}{2}\oa_0,\frac{1}{2}\ob_0\right\}_\ZZ$ for  $y_0= nc$; 
$\Ltily=\left\{\oa_0,\frac{1}{2}\ob_0\right\}_\ZZ=\MM$ for  $y_0= nc+y_1$; 
and $\Ltily=\left\{\oa_0,\frac{1}{2}(\oa_0+\ob_0)\right\}_\ZZ=\MM$ for  $y_0= nc+\frac{1}{2}c+y_1$,  $n\in\ZZ$; and $\Ltily=\MM$  for other values of $y_0$.

For $v_1=2\oa_0+\frac{1}{2}\ob_0$ then  $\Ltily=\left\{\frac{1}{2}\oa_0,\frac{1}{2}\ob_0\right\}_\ZZ$ for  $y_0= nc$; and $\Ltily=\left\{\oa_0,\frac{1}{2}(\oa_0+\ob_0)\right\}_\ZZ$ for  $y_0= nc+y_1$; 
and $\Ltily=\left\{\oa_0,\frac{1}{2}\ob_0 \right\}_\ZZ$ for  $y_0= nc+\frac{1}{2}c+y_1$,  $n\in\ZZ$; and $\Ltily=\MM$  for other values of $y_0$.

 \bigbreak
 
 The cases above are all the possibilities for $\Ltily=\MM$.
 \bigbreak
 
Note that from cases {\bf 1}, {\bf 3.1} and {\bf 4} it follows that if 
$\LL=\left\{\left(\oa_0,0\right),\left(\ob_0,0\right), \oc\right\}_\ZZ$ for some $\oc\in\RR^3$ then $\Ltily=\MM$ for all $y_0$, except in the cases  $\left(\left(v_1,y_1\right),\sigma\right)\in\Gamma$ 
with
 $v_1\in\left\{\frac{1}{2}\oa_0 ,\frac{1}{2}\ob_0 ,\frac{1}{2}\left(\oa_0 +\ob_0 \right)\right\}$.
 In all other cases we get different results for $\Ltily$ depending on $y_0$.


\section{Example --- projection into a hexagonal lattice}\label{SecProjectionToHex}
We specialise further, and suppose that, for some value of the width $y_0$, we have 
$$
\Ltily = \left\{\left(1,0\right),\left(\frac{1}{2},\frac{\sqrt{3}}{2}\right)\right\}_\ZZ=
 \left\{{\rm\bf h}_1,{\rm\bf h}_2\right\}_\ZZ=
\hex
$$
that we call the {\it standard hexagonal lattice}. We list all the possible lattices $\LL$ that lead to this result for some $y_0$ and describe what happens for other  values of $y_0$. The lattices are described up to symmetries of the form $\alpha_+$ where $\alpha$ belongs to the holohedry of $\hex$.
This particular example is interesting because of its connection to special patterns discussed in \cite{CLO,Gomes}.

We follow the numbering of cases of Section~\ref{secProjectGeneral} to identify the lattices $\LL$.
The projected lattices $\Ltily$ appear in Tables~\ref{Tabela2}--\ref{Tabela5}, except for cases {\bf 2.1}. and {\bf 2.2}. where details are given below.

\noindent {\bf case 1}.
 $\LL=\left\{\left(1,0,0\right),\left(\frac{1}{2},\frac{\sqrt{3}}{2},0\right), \left(\oc_0,c \right)\right\}_\ZZ$ for some $n\oc_0\not\in\hex$  for all $n\in\ZZ$ and  
$c\in \RR\backslash\{0\}$. Then $\Ltily=\hex$ for all $y_0$.

\noindent {\bf case 2}. $\sigma\LL\not\subset\LL$

\noindent {\bf case 2.1}.
$\LL=\left\{\left(1,0,a_3\right),\left(\frac{1}{2},\frac{\sqrt{3}}{2},b_3\right), \left(0,0,c \right)\right\}_\ZZ$ for some $c\ne 0$; and
$\sigma\LL\not\subset\LL$, hence  either $a_3\ne 0$ or $b_3\ne 0$.
Then for $y_0=nc$, $n\in\ZZ$ we have $\Ltily=\hex$. For other values of $y_0$, it is $\Ltily=\restr_0(\LL)\subsetneq\hex$, depending on 
$$
D(\LL)=\{(m_1,m_2)\in\ZZ^2: m_1a_3+m_2b_3=0 \: \pmod{b}\}.
$$
For instance, if $D(\LL)=\{(m_1,m_2)\}_\ZZ$ then $\restr_0(\LL)=\left\{m_1\left(1,0\right)+m_2\left(\frac{1}{2},\frac{\sqrt{3}}{2}\right)\right\}_\ZZ$.

\noindent {\bf case 2.2}. 
If $\sigma\LL\not\subset\LL$ and $(0,0,c)\in\LL$, $c\ne 0$, is minimal in its direction, then for $y_0=nc$, $n\in\ZZ$ we have $\Ltily=\proj(\LL)$, whereas for other values of $y_0$, it is $\Ltily=\restr_0(\LL)$.
If, moreover,
$D(\LL)=\left\{\left(m_1,m_2\right),\left(n_1,n_2\right)\right\}_\ZZ$,  for some linearly independent 
$\left(m_1,m_2\right),\left(n_1,n_2\right)\in\ZZ^2$,
then we may have  $\hex\subsetneq\Ltily$ for $y_0=nc$, $n\in\ZZ$ and  $\Ltily=\hex$ for the remaining $y_0$.
In the first case,
$\LL=\left\{\left(\na ,a_3\right), \left(\nb ,b_3\right), (0,0,c)\right\}_\ZZ$ with $c>0$, $a_3, b_3\in\left[0,c\right)$
and $\na ,\nb \in\plano$ satisfy $m_1\na+m_2\nb={\rm\bf h}_1$ and $n_1\na+n_2\nb={\rm\bf h}_2$.

\noindent {\bf case 3}.
$\sigma\LL\subset\LL$ and $\LL_\sigma\subset\LL$.

\noindent {\bf case 3.1}.
 $\LL=\left\{\left(1,0,0\right),\left(\frac{1}{2},\frac{\sqrt{3}}{2},0\right),  \left(0,0,c\right)\right\}_\ZZ$.

\noindent {\bf case 3.2}.
 $\LL=\left\{\left(1,0,\frac{c}{2}\right),\left(\frac{1}{2},\frac{\sqrt{3}}{2},0\right),  \left(0,0,c\right)\right\}_\ZZ$.

\noindent {\bf case 3.3}.
 $\LL=\left\{\left(\frac{1}{2},0,\frac{c}{2}\right),\left(\frac{1}{2},\frac{\sqrt{3}}{2},0\right),  \left(0,0,c\right)\right\}_\ZZ$.

\noindent {\bf case 4}.
 $\LL=\left\{\left(\na,0\right),\left(\nb,0\right),  \left(0,0,c\right)\right\}_\ZZ$
and $\left(\left(v_1,y_1\right),\sigma\right)\in\Gamma$, with 
$v_1\in\left\{ 0, \frac{1}{2}\na,  \frac{1}{2}\nb,  \frac{1}{2}(\na+\nb)\right\}$.

\noindent {\bf case 4.1}.
$\LL=\left\{\left(1,0,0\right),\left(\frac{1}{2},\frac{\sqrt{3}}{2},0\right),  \left(0,0,c\right)\right\}_\ZZ$.\\

\noindent {\bf case 4.2}. $v_1=\left(\frac{1}{2},0\right)$, and 
$\LL=\left\{\left(1,0,0\right),\left(\frac{1}{2},\frac{\sqrt{3}}{2},0\right),  \left(0,0,c\right)\right\}_\ZZ$. 

\noindent {\bf case 4.3}. $v_1=\left(\frac{1}{2},\frac{\sqrt{3}}{2}\right)$, and 
$\LL=\left\{\left(1,0,0\right),\left(1,\sqrt{3},0\right),  \left(0,0,c\right)\right\}_\ZZ$.

\noindent {\bf case 5}. 
 $\LL=\left\{\left(\na,\frac{c}{2}\right),\left(\nb,0\right),  \left(0,0,c\right)\right\}_\ZZ$ (without loss of generality)\\
and $\left(\left(v_1,y_1\right),\sigma\right)\in\Gamma$, 
 with 
$v_1\in\left\{ 0, \frac{1}{2}\na,  \frac{1}{2}\nb,  \frac{1}{2}(\na+\nb)\right\}$.

\noindent {\bf case 5.1}. 
$\LL=\left\{\left(1,0,\frac{c}{2}\right),\left(\frac{1}{2},\frac{\sqrt{3}}{2},0\right),  \left(0,0,c\right)\right\}_\ZZ$.

\noindent {\bf case 5.2}. 
$\LL=\left\{\left(1,0,\frac{c}{2}\right),\left(1,\sqrt{3},0\right),  \left(0,0,c\right)\right\}_\ZZ$.

\noindent {\bf case 5.3}. 
$\LL=\left\{\left(\frac{3}{4},-\frac{\sqrt{3}}{4},\frac{c}{2}\right),\left(\frac{1}{2},\frac{\sqrt{3}}{2},0\right),  \left(0,0,c\right)\right\}_\ZZ$.

\noindent {\bf case 5.4}. 
$\LL=\left\{\left(\frac{1}{2},0,\frac{c}{2}\right),\left(1,\sqrt{3},0\right),  \left(0,0,c\right)\right\}_\ZZ$.

\noindent {\bf case 5.5}. 
$\LL=\left\{\left(\frac{1}{2},0,\frac{c}{2}\right),\left(\frac{1}{2},\frac{\sqrt{3}}{2},0\right),  \left(0,0,c\right)\right\}_\ZZ$.

\bigbreak

The particular lattices $\Ltil\subset\RR^2$ that appear in this section are listed in Table~\ref{TableLtil}.

 \begin{table}[hht]
\begin{center}
\begin{tabular}{ccll}
$\oa$	&	$\ob$	&	name/type	&	symbol	\\ \hline
$\left(1,0\right)$	&	$\left(\frac{1}{2},\frac{\sqrt{3}}{2}\right)$	&	hexagonal	&	$\hex$	\\ \hline
$\left(\frac{1}{2},0\right)$	&	$\left(0,\frac{\sqrt{3}}{2}\right)$	&	rectangular I	&	rec I	\\
$\left(0,\frac{\sqrt{3}}{2}\right)$	&	$\left(1,0\right)$	&	rectangular II	&	rec II	\\
$\left(\frac{1}{4},\frac{\sqrt{3}}{4}\right)$	&	$\left(\frac{3}{2},-\frac{\sqrt{3}}{2}\right)$	&	rectangular III	&	rec III	\\ \hline
$\left(\frac{1}{4},\frac{\sqrt{3}}{4}\right)$	&	$\left(-\frac{3}{4},\frac{\sqrt{3}}{4}\right)$	&	rotated rectangular I	&	$R_{\pi/3}$ rec I	\\
$\left(\frac{1}{4},-\frac{\sqrt{3}}{4}\right)$	&	$\left(\frac{3}{4},\frac{\sqrt{3}}{4}\right)$	&	rotated rectangular I	&	$R_{-\pi/3}$ rec I	\\	\hline
\end{tabular}
 \end{center}
 \caption{Generators $\oa$, $\ob$ of lattices $\Ltily=\left\{\oa,\ob\right\}_\ZZ$ appearing as projections.
 The last two correspond to rotations of  rec I by an angle $\pi/3$, in opposite directions.\label{TableLtil}}
 \end{table}

For each of the cases above, we describe in Tables~\ref{Tabela2}--\ref{Tabela5} the lattices $\Ltily$ corresponding to a set of values of the width $y_0$ of the projection band. The relevant sets are 
$$
A=\left\{nc: n \in \NN \right\} \qquad
B=\left\{y_1+nc: n \in \NN \right\} \qquad
C=\left\{y_1+\frac{c}{2}+nc: n \in \NN \right\}
$$
where the value of $c$ correspondes to the generator $(0,0,.c)$ of $\LL$. We indicate by ``${\mathcal O}$'' ({\em other values}) 
the set of all values of $y_0\in \RR^+$ outside the union of the other sets in the same column.

\begin{table}[htt]
\begin{center}
\begin{tabular}{r||c|c|c|c|c|c}
$\Ltily\backslash$ cases	&	1	&	2.1	&	2.2	&	3.1	&	3.2	&	3.3	\\ \hline\hline
$\hex$	&	$\RR^+$	&	$A$	&	$\OO$	&	$\RR^+$	&	$A$	&	$\OO$	\\
rec I	&	 ---	&	 ---	&	 ---	&	 ---	&	 ---	&	$A$	\\
2 $R_{\pi/3}$ rec I	&	 ---	&	 ---	&	 ---	&	 ---	&	$\OO$	&	 ---	\\
sublattice	&	 ---	&	$\OO$	&	 ---	&	 ---	&	 ---	&	 ---	\\
superlattice	&	 ---	&	 ---	&	$A$	&	 ---	&	 ---	&	 ---		
\end{tabular}	
\end{center}
\caption{Lattices $\Ltily$ in the projection for  cases 1--3.3 above, with the conventions of Table~\ref{TableLtil}. We  indicate by ``sublattice'' a submodule of the  hexagonal lattice $\hex$, either $\{(0,0)\}$ or with  one or two generators. By ``superlattice'' we mean a lattice strictly containing $\hex$ as a submodule.
Also ``2 $R_{\pi/3}$ rec I'' stands for the lattice $\left\{2\oa,2\ob\right\}_\ZZ$ 
where $\oa$, $\ob$ are generators for ``$R_{\pi/3}$ rec I'', the rotated rectangular lattice. \label{Tabela2}}
\end{table}

\begin{table}[htt]
\begin{center}
\begin{tabular}{r||c|c|c|c|c|c} 
cases	&	4.1	&	4.1	&	4.1	&	4.1	&	4.2	&	4.3	\\
$\Ltily\backslash$ $v_1$	&	$\left(0,0\right)$	&	$\left(\frac{1}{2},0\right)$	&	$\left(\frac{1}{4},\frac{\sqrt{3}}{4}\right)$	&	$\left(\frac{3}{2},\frac{\sqrt{3}}{2}\right)$	&	$\left(\frac{1}{2},0\right)$	&	$\left(\frac{1}{2},\frac{\sqrt{3}}{2}\right)$	\\ \hline \hline
$\hex$	&	$\RR^+$	&	$\OO$	&	$\OO$	&	$\OO$	&	$A\cup B$	&	$A\cup B$	\\
rec I	&	 ---	&	$A\cup B$	&	 ---	&	 ---	&	 ---	&	 ---	\\
2 rec I	&	 ---	&	 ---	&	 ---	&	$A\cup B$	&	 ---	&	$\OO$	\\
$R_{\pi/3}$ rec I	&	 ---	&	 ---	&	$A\cup B$	&	 ---	&	 ---	&	 ---	\\
2$R_{\pi/3}$ rec I	&	 ---	&	 ---	&	 ---	&	 ---	&	$\OO$	&	 ---	
\end{tabular}	
\end{center}
\caption{Lattices $\Ltily$ in the projection for  case 4 above, with the conventions of Table~\ref{TableLtil}. We  indicate by ``2 $R_{\pi/3}$ rec I''  the lattice $\left\{2\oa,2\ob\right\}_\ZZ$ 
where $\oa$, $\ob$ are generators for ``$R_{\pi/3}$ rec I'', the rotated rectangular lattice, with a similar convention for the other lattices. \label{Tabela3}}
\end{table}

\begin{table}[htt]
\begin{center}
\begin{tabular}{r||c|c|c|c|c|c}
\phantom{$\Ltily\backslash$} cases	&	5.1	&	5.1	&	5.2	&	5.2	&	5.3	&	5.3	\\
$\Ltily\backslash$ $v_1$	&	$\left(0,0\right)$	&	$\left(1,0\right)$	&	$\left(\frac{1}{2},\frac{\sqrt{3}}{2}\right)$	&	$\left(\frac{3}{2},\frac{\sqrt{3}}{2}\right)$	&	$\left(\frac{1}{4},\frac{\sqrt{3}}{4}\right)$	&	$\left(1,0\right)$	\\ \hline \hline
$\hex$	&	$A\cup C$	&	$A\cup B$	&	$A$	&	$A$	&	$C$	&	$B$	\\
$2\hex$	&	 ---	&	 ---	&	$\OO$	&	$\OO$	&	 ---	&	 ---	\\
rec III	&	 ---	&	 ---	&	 ---	&	 ---	&	$B$	&	$C$	\\
$R_{\pi/3}$ rec I	&	 ---	&	 ---	&	 ---	&	 ---	&	$A$	&	$A$	\\
$2 R_{\pi/3}$ rec I	&	 ---	&	 ---	&	$C$	&	$B$	&	 ---	&	 ---	\\
$R_{-\pi/3}$ rec I	&	$\OO$	&	$\OO$	&	$B$	&	$C$	&	$\OO$	&	$\OO$	
\end{tabular}	
\end{center}
\caption{Lattices $\Ltily$ in the projection for  cases 5.1--5.3 above, with the conventions of Table~\ref{TableLtil}. We  indicate by  ``2 $R_{\pi/3}$ rec I''   the lattice $\left\{2\oa,2\ob\right\}_\ZZ$ 
where $\oa$, $\ob$ are generators for ``$R_{\pi/3}$ rec I'', the rotated rectangular lattice, with a similar convention for the other lattices. \label{Tabela4}}
\end{table}

\begin{table}[htt]
\begin{center}
\begin{tabular}{r||c|c|c|c|c|c}
\phantom{$\Ltily\backslash$} cases	&	5.4	&	5.4	&	5.5	&	5.5	&	5.5	&	5.5	\\
$\Ltily\backslash$ $v_1$	&	$\left(\frac{1}{2},\frac{\sqrt{3}}{2}\right)$	&	$\left(1,\frac{\sqrt{3}}{2}\right)$	&	$\left(0,0\right)$	&	$\left(\frac{1}{2},0\right)$	&	$\left(\frac{1}{4},\frac{\sqrt{3}}{4}\right)$	&	$\left(\frac{3}{4},\frac{\sqrt{3}}{4}\right)$	\\ \hline \hline
$\hex$	&	$B$	&	$C$	&	$\OO$	&	$\OO$	&	$\OO$	&	$\OO$	\\
$\frac{1}{2} \hex$ 	&	 ---	&	 ---	&	 ---	&	 ---	&	$A$	&	$A$	\\
rec I	&	$A$	&	$A$	&	$A\cup C$	&	$A\cup B$	&	 ---	&	 ---	\\
2 rec I	&	$\OO$	&	$\OO$	&	 ---	&	 ---	&	 ---	&	 ---	\\
rec II	&	$C$	&	$B$	&	 ---	&	 ---	&	 ---	&	 ---	\\
$R_{\pi/3}$ rec I	&	 ---	&	 ---	&	 ---	&	 ---	&	$B$	&	$C$	\\
$R_{-\pi/3}$ rec I	&	 ---	&	 ---	&	 ---	&	 ---	&	$C$	&	$B$	
\end{tabular}	
\end{center}
\caption{Lattices $\Ltily$ in the projection for  cases 5.4 and 5.5 above, with the conventions of Table~\ref{TableLtil}. We  indicate by  ``$\frac{1}{2} \hex$'' the lattice $\left\{\frac{1}{2}\oa,\frac{1}{2}\ob\right\}_\ZZ$ 
where $\oa$, $\ob$ are generators for $\hex$, with a similar convention for the other lattices. \label{Tabela5}}
\end{table}

\section{Discussion of the example}\label{SecDiscussion}
For most lattices, only  some special projection widths yield a hexagonal lattice.
The cases where for all the projection widths there is a  three-fold rotation in the group of $\Pi_{y_0}(X_\Gamma)$ have been discussed in \cite{Oliveira}. 
These correspond to the primitive, the body-centered and face-centered cubic lattices, the rhombohedral lattice and the hexagonal lattice in 3 dimensions, always in special positions. 
In this section we compare this to our results.

Inspection of Tables~\ref{TableLtil}--\ref{Tabela5} shows that the projection is the hexagonal lattice for all the band widths in cases 1 and 3.1,  in case 4.1 with $v_1=(0,0)$, and in some situations in case 2.

The three dimensional hexagonal lattice corresponds to case 3.1, if the point group of $\Gamma$ does not contain  $\sigma$,  it corresponds to case 4.1 with $v_1=\left(0,0\right)$ if $\sigma\in \ptg$.

The other four possibilities correspond to special situations in case 2, where the set $D(\LL)$ has two independent generators, so $\restr_0(\LL)$ is a two-dimensional hexagonal lattice. Details are given in Tables~\ref{TabelaCompara} and \ref{TabelaCaso2}.

\begin{table}[htt]
\begin{center}
\begin{tabular}{l|l|l}
$\LL$	&	$\LL$	&	$D(\LL)$ 	\\
lattice	&	generators	&	generators	\\ \hline
rhombohedral	&	$\left(1,0,1\right)$,  $\left(\frac{-1}{2},\frac{-\sqrt{3}}{2},1\right)$, $\left(0,0,3\right)$	&	$\left(2,1\right)$, $\left(1,2\right)$	\\
primitive cubic	&	$\left(1,0,0\right)$, $\left(\frac{1}{2},\frac{\sqrt{3}}{6},\frac{-\sqrt{6}}{6}\right)$, $\left(0,0,\frac{\sqrt{6}}{2}\right)$	&	$\left(0,3\right)$, $\left(1,3\right)$	\\
body centered cubic	&	$\left(1,0,0\right)$, $\left(\frac{1}{2},\frac{\sqrt{3}}{6},\frac{-\sqrt{6}}{12}\right)$, $\left(0,0,\frac{\sqrt{6}}{4}\right)$	&	$\left(0,3\right)$, $\left(1,3\right)$	\\
face centered cubic	&	$\left(1,0,0\right)$, $\left(\frac{1}{2},\frac{\sqrt{3}}{6},\frac{\sqrt{6}}{3}\right)$, $\left(0,0,\sqrt{6}\right)$	&	$\left(0,3\right)$, $\left(1,3\right)$	
\end{tabular}	
\end{center}
\caption{Lattices in case 2 that project into a hexagonal lattice for all projection widths\label{TabelaCompara}}
\end{table}

\begin{table}[htt]
\begin{center}
\begin{tabular}{l|l|l}
$\LL$	&	$\proj(\LL)$	&	$\restr_0(\LL)$	\\
lattice	&	generators	&	generators	\\ \hline
rhombohedral	&	 $\left(1,0\right)$, $\left(\frac{1}{2},\frac{\sqrt{3}}{2}\right)$ 	&	$\left(\frac{3}{2},\frac{\sqrt{3}}{2}\right)$, $\left(\frac{3}{2},\frac{-\sqrt{3}}{2}\right)$	\\
all cubic	&	$\left(\frac{1}{2},\frac{\sqrt{3}}{6}\right)$, $\left(\frac{1}{2},\frac{-\sqrt{3}}{6}\right)$	&	 $\left(1,0\right)$, $\left(\frac{1}{2},\frac{\sqrt{3}}{2}\right)$ 	\end{tabular}	
\end{center}
\caption{In case 2 we have $\sigma\LL\not\subset\LL$, hence $\Ltily=\proj(\LL)$ if $\left(0,y_0\right)\in\LL$ and $\Ltily=\restr_0(\LL)$ otherwise. In these special cases, projection always yields a hexagonal lattice, but the size of the cell varies.\label{TabelaCaso2}}
\end{table}

\begin{figure}[hht]
\includegraphics[scale=0.5]{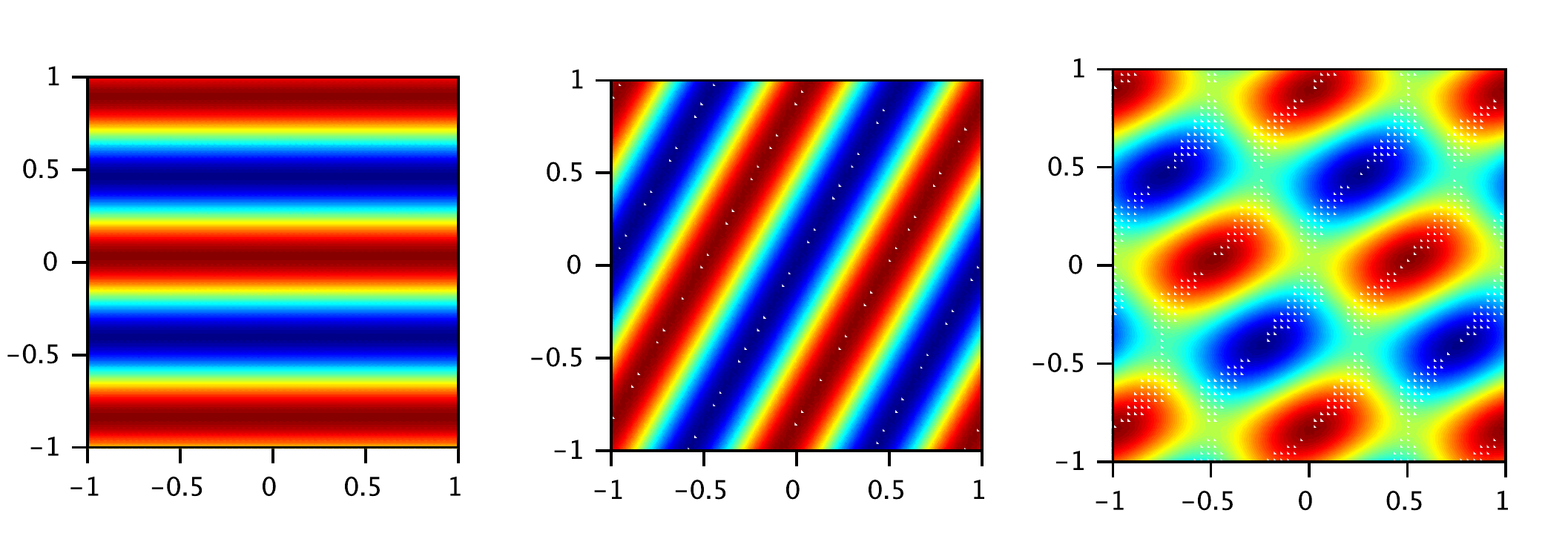}\\
\caption{Level sets for  projections of functions periodic under the  triclinic lattice $\LL_T$ without symmetry (case I), for $y_0=\frac{\sqrt{2}}{14}$.
Left to right:  level sets for the projections of  $\omega_{k_2}$, of $\omega_{k_3}$ and of $\omega_{k_2}+10\omega_{k_3}$.
The only common symmetry of the patterns is the $\Ltily$-periodicity.
}
\label{FigTriclinicNoSymmetry}
\end{figure}

The triclinic lattices in case 1 do not appear on the list in \cite{Oliveira} because their symmetry group does not contain any rotation. 
Generically, the holohedry of a lattice in case 1 is trivial.
An example of projection of invariant functions is given in Figure~\ref{FigTriclinicNoSymmetry} for 
$$
\LL_T=\left\{ \left(1,0,0\right),  \left(\frac{1}{2},\frac{\sqrt{3}}{2},0\right),  \left(1,1,1\right)\right\}_\ZZ
$$
with 
$$
\Ld_T=\left\{ k_1=\left(0,0,1\right),  k_2=\left(0,\frac{2\sqrt{3}}{3},\frac{-2\sqrt{3}}{3}\right),  k_3=\left(1,\frac{-\sqrt{3}}{3},\frac{\sqrt{3}}{3}-1\right)\right\}_\ZZ.
$$
The lattice of periods of the projected functions has holohedry $D_6$, but the $\frac{2\pi}{6}$ rotation is not a symmetry of the projected functions.

In the special case when  $\oc_0$ is parallel to a mirror of $\hex$, the holohedry of a triclinic lattice in case I may contain a reflection. 
An example is the lattice 
$$
\LL_S=\left\{ \left(1,0,0\right),  \left(\frac{1}{2},\frac{\sqrt{3}}{2},0\right),  \left(\frac{\sqrt{2}}{2},0,1\right)\right\}_\ZZ
$$
with 
$$
\Ld_S=\left\{ \ell_1=\left(0,0,1\right),  \ell_2=\left(0,\frac{2\sqrt{3}}{3},0\right),  \ell_3=\left(1,\frac{-\sqrt{3}}{3},\frac{-\sqrt{2}}{2}\right)\right\}_\ZZ ,
$$
that is symmetric under the reflection $\gamma_y$ on the  plane orthogonal to $\left(0,1,0\right)$.
The $\LL_S$-periodic functions $f_{\ell_3}(x,y,z)=\omega_{\ell_3}(x,y,z)+\omega_{\gamma_y\ell_3}(x,y,z)$, and
 $f_{\ell_2}(x,y,z)=\omega_{\ell_2}(x,y,z)+\omega_{\gamma_y\ell_2}(x,y,z)$ are also $\gamma_y$-invariant, and hence their projections are also invariant for the reflection on the $\left(x,0\right)$ axis, as in Figure~\ref{FigTriclinicSymmetric}.

\begin{figure}
\includegraphics[scale=0.5]{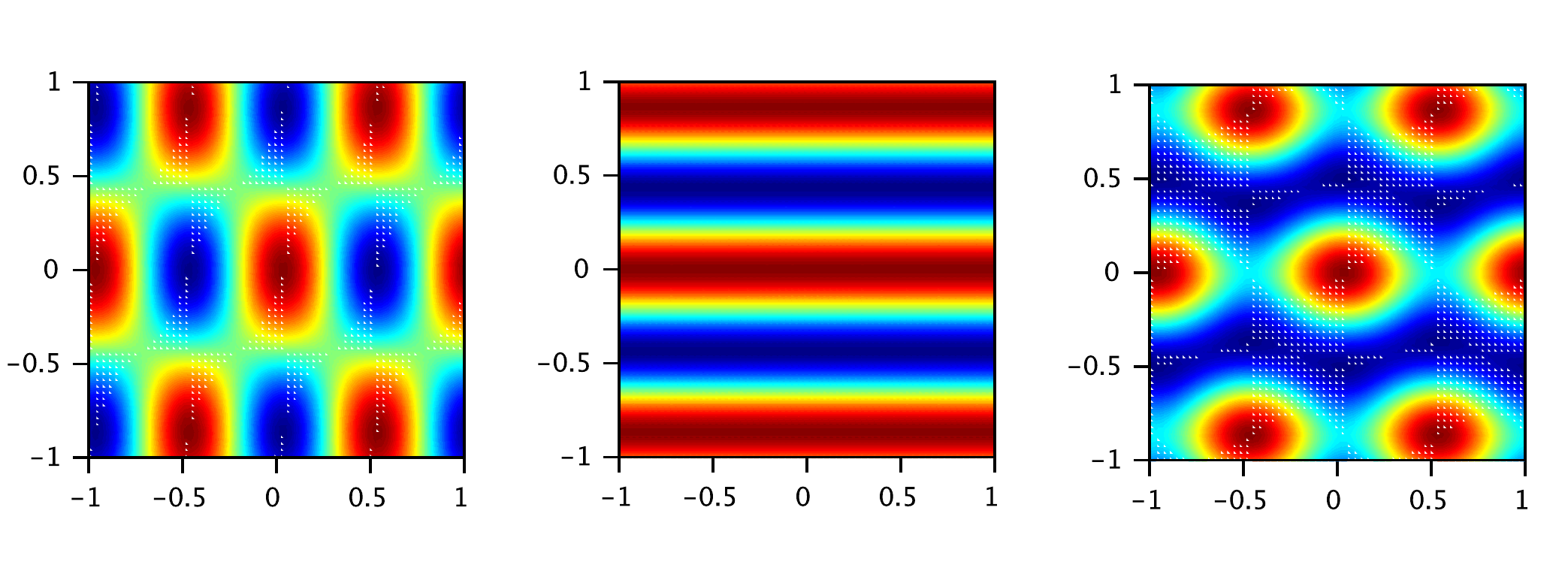}\\
\caption{Level sets for  projections of functions with reflection symmetry and periodic under the  triclinic lattice  with symmetry $\LL_S$ (case I), for $y_0=\frac{\sqrt{2}}{14}$.
Left to right:  level sets for the projections of $f_{\ell_3}=\omega_{\ell_3}+\omega_{\gamma_y\ell_3}$, 
of $f_{\ell_2}=\omega_{\ell_2}+\omega_{\gamma_y\ell_2}$, and of $f_{\ell_2}+f_{\ell_3}$.
The patterns are symmetric under reflection on the horizontal axis, but not under rotation by $\frac{2\pi}{6}$.
}
\label{FigTriclinicSymmetric}
\end{figure}

The projections of triclinic lattices above are good illustrations of the fact that the symmetries of the lattice of periods are not necessarily symmetries of the pattern. Projected patterns with period lattice $\hex$ and with $D_6$-symmetry are exhibited in \cite{CLO} together with a general method for finding them.

\section*{Acknowledgements}
{\small
CMUP (UID/MAT/00144/2013) iis funded by FCT (Portugal) with national (MEC) and European structural funds (FEDER), under the partnership agreement PT2020. 
%
%
E. M. Pinho was partly supported by  the grant
SFRH/BD/13334/2003 of FCT and by UBI-Universidade
da Beira Interior, Portugal.}


\renewcommand{\bibname}{References}

\end{document}